\documentclass[11pt]{article}

\usepackage{amsmath, amsthm, amssymb, graphics, graphicx}
\usepackage{hyperref}
\usepackage{cleveref}
\usepackage[none]{hyphenat}
\usepackage{tocloft}
\usepackage{cancel}
\usepackage{multicol}

\usepackage{tikz-cd}

\DeclareMathOperator{\lastBorn}{\ell}
\DeclareMathOperator{\choosePath}{c}
\DeclareMathOperator{\nextB}{\textsc{next-b}}
\DeclareMathOperator{\nextT}{\textsc{next-t}}
\DeclareMathOperator{\im}{im}

\usepackage[all,cmtip]{xy}
\DeclareMathOperator{\interior}{^{\circ}}

\DeclareMathOperator{\desc}{\prec}

\DeclareMathOperator{\adj}{\curvearrowright}

\DeclareMathOperator{\Rl}{R}
\DeclareMathOperator{\notRl}{\cancel{\Rl}}
\DeclareMathOperator{\Gr}{\Gamma}

\DeclareMathOperator{\Xspan}{X-span}
\DeclareMathOperator{\Yspan}{Y-span}

\DeclareMathOperator{\score}{s}
\DeclareMathOperator{\antiscore}{\overline{s}}

\newcommand{\sm}{\setminus}

\binoppenalty=\maxdimen
\relpenalty=\maxdimen
\usepackage{indentfirst}
\usepackage{enumitem}
\setlist[itemize]{noitemsep, topsep=0pt}

\newtheorem{theorem}{Theorem}[section]
\newtheorem{lemma}[theorem]{Lemma}

\newtheorem{definition}[theorem]{Definition}
\newtheorem{statement}{Statement}

\title{Burling graphs revisited, part I:\\New characterizations}
\author{Pegah Pournajafi\thanks{Univ Lyon, EnsL, UCBL, CNRS, LIP,
    F-69342, LYON Cedex 07, France. Partially supported by the LABEX
    MILYON (ANR-10-LABX-0070) of Universit\'e de Lyon, within the
    program ‘‘Investissements d'Avenir’’ (ANR-11-IDEX-0007) operated
    by the French National Research Agency (ANR) and by Agence
    Nationale de la Recherche (France) under research grant ANR
    DIGRAPHS ANR-19-CE48-0013-01.}~~and Nicolas Trotignon\footnotemark[1]}
\date{\today}
\begin{document}
	
       \maketitle
	
       \begin{abstract}
         The Burling sequence is a sequence of triangle-free graphs of
         increasing chromatic number. Each of them is isomorphic to
         the intersection graph of a set of axis-parallel boxes in
         $\mathbb{R}^3$. These graphs were also proved to have other
         geometrical representations: intersection graphs of line
         segments in the plane, and intersection graphs of frames,
         where a frame is the boundary of an axis-aligned rectangle in
         the plane.

         We call Burling graph every graph that is an induced subgraph
         of some graph in the Burling sequence.  We give five new
         equivalent ways to define Burling graphs. Three of them are
         geometrical, one is of a more graph-theoretical flavor, and
         one, that we call abstract Burling graphs, is more axiomatic.
       \end{abstract}
       
\providecommand{\keywords}[1]
{
  \small	
  {\textit{Keywords:}} #1
}

       
       \section{Introduction}

       Graphs in this paper have neither loops nor multiple edges. In
       this introduction they are non-oriented, but oriented graphs
       will be considered in the rest of the paper.  A class of graphs
       is \emph{hereditary} if it is closed under taking induced
       subgraphs.  A \emph{triangle} in a graph is a set of three
       pairwise adjacent vertices, and a graph is \emph{triangle-free}
       if it contains no triangle.  The \emph{intersection graph} of
       sets $S_1, \dots, S_n$ is defined as follows: the vertices are
       the sets, and for $ i \neq j $, $ S_i $ is connected to $ S_j $
       by an edge if and only if $S_i \cap S_j \neq \varnothing$.

       \subsection*{Burling graphs}
       
       In 1965, Burling~\cite{Burling65} proved that triangle-free
       intersection graphs of axis-aligned boxes in $\mathbb{R}^3$
       have unbounded chromatic number. The work of Burling uses
       purely geometric terminology. However, one may rephrase it into
       a more modern graph theoretic setting and define in a
       combinatorial way a sequence of triangle-free graphs with
       increasing chromatic number, and prove that each of them is
       isomorphic to the intersection graph of a set of axis-parallel
       boxes in $\mathbb{R}^3$.  The definition of this sequence is
       recalled in Section~\ref{sec:BequalD}, and we call it the
       \emph{Burling sequence}.

       It was later proved that every graph in the Burling sequence is
       isomorphic to the intersection graph of various geometrical
       objects.  In a work of Pawlik, Kozik, Krawczyk, Lason, Micek,
       Trotter and Walczak~\cite{Pawlik2012Jctb}, the objects under
       consideration are line segments in the plane.  The intersection
       graphs of line segments in the plane are called \emph{line
         segment graphs}.  In fact, the Burling sequence was
       rediscovered in~\cite{Pawlik2012Jctb}, and the way it is
       presented in recent works is inspired by this paper.

       In works of Pawlik, Kozik, Krawczyk, Lason, Micek, Trotter and
       Walczak~\cite{DBLP:journals/dcg/PawlikKKLMTW13}, Krawczyk,
       Pawlik and Walczak~\cite{Krawczyk2015}, and later Chalopin,
       Esperet, Li and Ossona de Mendez~\cite{Chalopin2014}, the
       objects under consideration are frames, where a \emph{frame} is
       the boundary of an axis-aligned rectangle in the plane. In
       fact, a stronger result is given in~\cite{Krawczyk2015}: it is
       proved that every graph of the Burling sequence is a
       \emph{restricted frame graph}, meaning that the frames satisfy
       several constraints that we recall in Section~\ref{sec:geom}.

       The Burling sequence also attracted attention lately because it
       is a good source of examples of graphs of high chromatic number
       in some hereditary classes of graphs that are not defined
       geometrically, but by excluding several patterns as induced
       subgraphs.  Most notably, it is proved in~\cite{Pawlik2012Jctb}
       that they provide a counter-example to a well-studied
       conjecture of Scott, see~\cite{Scott18} for a survey.
        
       Since graphs of the Burling sequence appear in the context of
       hereditary classes of graphs, it is natural to define \emph{Burling
         graphs} as graphs that are induced subgraphs of some graph in
       the Burling sequence.  Observe that Burling graphs trivially
       form a hereditary class (in fact, the smallest hereditary class
       that contains the Burling sequence).  The goal of this work is
       a better understanding of Burling graphs.

       \subsection*{New geometrical characterizations}
       
       In this first part, we give three new characterizations of
       Burling graphs: as intersection graphs of frames, as
       intersection graphs of line segments in the plane and as
       intersection graphs of boxes of the 3-dimensional space.  The
       new feature of our characterizations is that they provide
       the following equivalences.  We put some restrictions on the geometrical
       objects, to obtain what we call \emph{strict frame graphs},
       \emph{strict line segment graphs} and \emph{strict box graphs}.
       The precise definitions are given in Section~\ref{sec:geom}.
       We then prove that a graph $G$ is a Burling graph \emph{if and
         only if} it is a strict frame graph (resp.\ a strict line
       segment graph, a strict box graph).

       Observe that in~\cite{Pawlik2012Jctb} (resp.\
       \cite{Chalopin2014,DBLP:journals/dcg/PawlikKKLMTW13,Pawlik2012Jctb}),
       it is proved that every Burling graph is a strict line segment
       graph (resp.\ a strict frame graph). The proofs there are
       implicit because the authors of these works do not mention our
       new restrictions, but it is straightforward to check that their
       way to embed Burling graphs in the plane satisfies them. Note
       that as pointed out by an anonymous referee, strict frame
       graphs are implicitly defined in~\cite{Krawczyk2015} as
       intersection graphs of clean directed families of frames
       avoiding some configurations.  Our contribution is the
       definition of the new restrictions on the geometrical objects
       and the converse statements: every strict frame graph and every
       strict line segment graph is a Burling graph.  Examples of line
       segment graphs that are not Burling graphs are already given
       in~\cite{Chalopin2014}. We go further by providing examples of
       restricted frame graphs that are not Burling graphs, showing
       that our new restrictions are necessary, see
       Figures~\ref{pic:k5sd-rfg}, \ref{pic:necklace-rfg}
       and~\ref{pic:wheel-rfg}.  Note that proving that the examples
       are not Burling graphs is non-trivial and postponed to the
       second part of this work~\cite{Part2}.

       \subsection*{Combinatorial characterizations}
       
       The definition of Burling graphs as induced subgraphs of graphs
       in the Burling sequence is not very easy to handle, at least for
       us. So, to prove the equivalence between Burling graphs and our
       geometrical constructions, we have to introduce two other new
       equivalent definitions of Burling graphs.  The first one,
       called \emph{derived graphs}, is purely combinatorial: we see
       how every Burling graph can be \emph{derived} from some tree
       structure using several simple rules (and we prove that only
       Burling graphs are obtained).  Derived graphs are defined in
       Section~\ref{section:Derived-graphs} and their equivalence with
       Burling graphs is proved in Section~\ref{sec:BequalD}.  In fact,
       derived graphs have a natural orientation that is very useful
       to consider, so they are defined as oriented graphs.
       Then, in Section~\ref{sec:abs} we prove that derived graphs can
       be defined as graphs obtained from a set with two relations
       satisfying a small number of axioms, again with simple rules.
       We call these \emph{abstract Burling graphs}.

       The advantage of this approach is that derived graphs seem to
       be specific and well structured, while abstract Burling graphs
       seem to be general (though they are equivalent).  As a
       consequence, derived graphs turn out to be useful to study the
       structure of Burling graphs, and this will be mostly done in
       the second part of this work~\cite{Part2}.  On the other hand, since abstract Burling
       graphs are ``general'', it is easy to check that geometrical
       objects satisfy the axioms in their definition.  So the proof
       that every graph arising from one of our geometrical
       characterizations is an abstract Burling graph, and therefore a
       Burling graph, is not too long. This is done in
       Section~\ref{sec:geom}.  Moreover, abstract Burling graphs
       might be of use to prove other geometrical or combinatorial
       characterizations of Burling graphs.

\begin{figure}[h!] \vspace*{-.1cm}  
	\centering
	\begin{tikzcd}[row sep=1.5cm, column sep=1cm] 
	&
	\parbox{5cm}{\centering Strict line segment graphs \\ Section 6} 
	\arrow[rd, hook, sloped, "Theorem \ 6.12"] \\
	\parbox{2.5cm}{\centering { \ } \\ Burling graphs \\ Section 4}
	\arrow[ru, hook, sloped, "{[4]}\  (See\  Lemma \ 6.11.)"]
	\arrow[r, hook, shift left, "Theorem \ 4.9"]
	\arrow[rd, hook,  sloped, "{[3]}\  (See\  Lemma \ 6.5.)"]
	& \parbox{3cm}{\centering Derived graphs \\ Section 3}
	\arrow[l, hook, shift left]
	\arrow[r, hook, shift left, sloped, "Theorem \ 5.7"]
	&
	\parbox{3cm}{\centering Abstract\\ Burling graphs \\ Section 5}
	\arrow[l, hook, shift left] \\
	&
	\parbox{3.5cm}{\centering Strict frame graphs \\ Section 6}
	\arrow[ru, hook, sloped, "Theorem \ 6.6"]
	\arrow[d, hook, shift left, "Theorem \ 6.15"] \\
	& 
	\parbox{3cm}{\centering Strict box graphs \\ Section 6}
	\arrow[u, hook, shift left] \\
	\end{tikzcd} \vspace*{-1.5cm}
	\caption{The six equivalent classes of graphs in this
		article. Each arrow shows an inclusion and the label of the
		arrow shows where the proof may be found.}  \label{fig:summary-partI}
\end{figure}
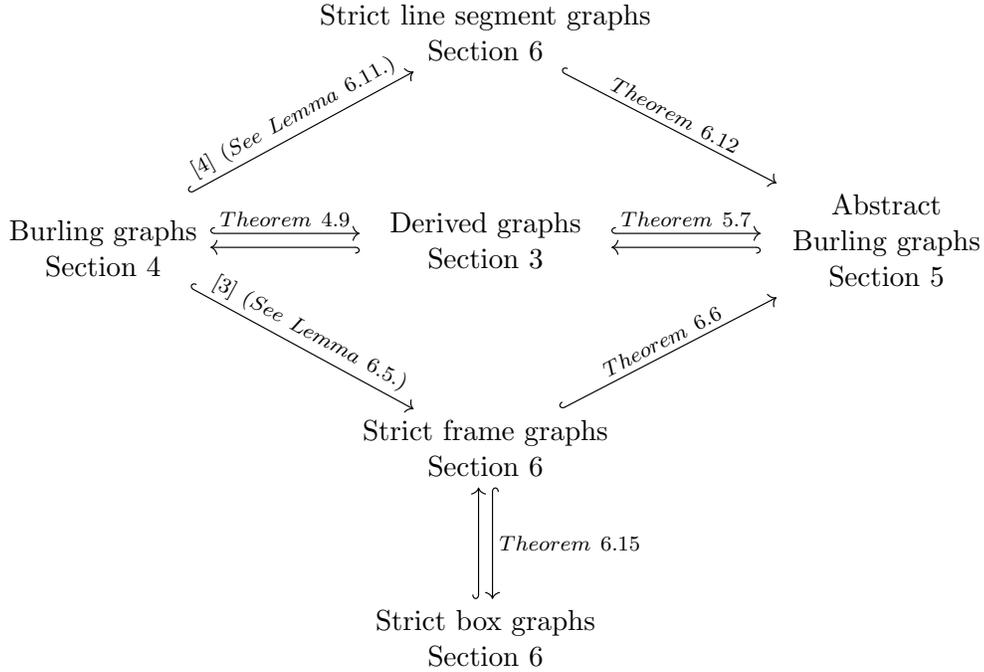

       \subsection*{Sum up}
       
       To sum up, we have now six different ways to define Burling
       graphs: as induced subgraphs of graphs in the classical Burling
       sequence, as derived graphs, as abstract Burling graphs, as
       strict frame graphs, as strict line segment graphs and as strict
       box graphs.  In Figure~\ref{fig:summary-partI}, we sum up where
       the different steps of the proofs of the equivalence between
       all classes can be found.

       The second~\cite{Part2} and third~\cite{Part3} parts of this work 
       are about the structure of Burling graphs and are motivated by the
       chromatic number in hereditary classes of graphs.  For
       instance, we prove a decomposition theorem for oriented Burling
       graphs and study under what conditions simple operations such
       as gluing along a clique, subdividing an edge or contracting an
       edge preserve being a Burling graph. We prove that no
       subdivision of $K_5$ is a Burling graph. Moreover, we classify as
       Burling or not Burling many series-parallel graphs. 

       We also prove that no wheel is a Burling graph, where a wheel
       is a graph made of a chordless cycle and a vertex with at least
       three neighbors in the cycle.  This last result was already
       known by Scott and Seymour (personal communication, 2017).
       Very recently, Davies rediscovered this independently and
       published a proof~\cite{davies2021trianglefree}.  Some of the
       results of Part 2 and Part 3 appeared in the master thesis of the first author, see~\cite{report}.

	\section{Notation}
        \label{sec:notation}

	There is a difficulty regarding notations in this paper. The graphs we
	are interested in will be defined from trees. More specifically, a
	tree $T$ is considered and a graph $G$ is \emph{derived} from it,
	following some rules defined in the next section.  We have $V(G) = V(T)$
	but $E(G)$ and $E(T)$ are different (disjoint, in fact). Also,
	even if we are originally motivated by non-oriented graphs, it turns
	out that $G$ has a natural orientation, and considering this orientation is
	essential in many of our proofs.
	
	So, in many situations, we have to deal simultaneously with the tree,
	the oriented graph derived from it and the underlying graph of this
	oriented graph. A last difficulty is that since we are interested in
	hereditary classes, we allow removing vertices from $G$. However, we have
	to keep all vertices of $T$ to study $G$ because of the so-called
	\emph{shadow vertices}: the vertices of $T$ that are not in $G$, which nevertheless capture
	essential structural properties of $G$.  All this will become clearer in
	the next section. For now, it explains why we need to be very careful
	about the notation that is mostly classical, see~\cite{BondyMurty}.

	\subsection*{Notation for trees}
	
	A \emph{tree} is a graph $T$ such that for every pair of vertices
	$u, v\in V(T)$, there exists a unique path from $u$ to $v$.  A
	\emph{rooted tree} is a pair $(T, r)$ such that $T$ is a tree and
	$r\in V(T)$.  The vertex $r$ is called the \emph{root} of $(T, r)$.
	Often, the rooted tree $(T, r)$ is abusively referred to as $T$, in
	particular when $r$ is clear from the context.
	
	In a rooted tree, each vertex $v$ except the root has a unique
	\emph{parent} which is the neighbor of $v$ in the unique path from the
	root to $v$. We denote the parent of $ v $ by $ p(v)$. 
	If $u$ is the parent of $v$, then $v$ is a \emph{child}
	of $u$.  A \emph{leaf} of a rooted tree is a vertex that has no children.
	Note that every tree has at least one leaf. We denote by $L(T)$ the
	set of all leaves of $T$.
	
	A \emph{branch} in a rooted tree is a path $ v_1 v_2 \dots v_k $ such
	that for each $ i \in \{1, \dots, k-1\} $, the vertex~$ v_i$ is the parent of $v_{i+1}$.  This
	branch \emph{starts} at $v_1$ and \emph{finishes} at $v_k$.  A branch
	that starts at the root and finishes at a leaf is a \emph{principal}
	branch.  Note that every rooted tree has at least one principal
	branch.
	
	If $T$ is a rooted tree, the \emph{descendants} of a vertex $v$ are
	all the vertices that are on a branch starting at $v$. The
	\emph{ancestors} of $v$ are the vertices on the unique path from $v$
	to the root of~$T$.  Notice that a vertex is a descendant and an
	ancestor of itself. Any descendant of a vertex $ v$, other than itself, is called a \emph{proper descendant} of $ v $.

	It is classical to orient the edges of a rooted tree (from the root,
	or sometimes to the root), but to avoid any confusion with the
	oriented graph derived from a tree, we will not use any of these
	orientations here. Also, we will no more use words such as
	\emph{neighbors}, \emph{adjacent}, \emph{path}, etc.\ for trees. Only
	\emph{parent}, \emph{child}, \emph{branch}, \emph{descendant} and
	\emph{ancestor} will be used.

	\subsection*{Notation for graphs and oriented graphs}
	
	By \emph{graph}, we mean a non-oriented graph with no loops and no
	multiple edges.  By \emph{oriented graph}, we mean a graph whose edges
	(called \emph{arcs}) are all oriented, and no arc is oriented in
	both directions.  When $G$ is a graph or an oriented graph, we denote
	by $V(G)$ its vertex set.  We denote by $E(G)$ the set of edges of a
	graph $G$ and by $A(G)$ the set of arcs of an oriented graph $ G $.  When
	$u$ and $v$ are vertices, we use the same notation $uv$ to denote an
	edge and an arc.  However, observe that the arc $uv$ is different from the arc $vu$,
	while the edge $uv$ is equal to the edge $vu$.
	
	For an oriented graph $ G $, its \emph{underlying graph} is the graph
	$H$ such that $V(H) = V(G)$ and for all $u, v\in V(H)$, $uv\in E(H)$
	if and only if $uv\in A(G)$ or $vu\in A(G)$.  We then also say that
	$G$ is an \emph{orientation of $H$}.  When there is no risk of
	confusion, we often use the same letter to denote an oriented
	graph and its underlying graph.
	
	In the context of oriented graphs, we use the words
	\emph{in-neighbor}, \emph{out-neighbor}, \emph{in-degree},
	\emph{out-degree}, \emph{sink} and \emph{source} with their classical meaning.
	Terms from the non-oriented realm, such as \emph{degree},
	\emph{neighbor}, \emph{isolated vertex} or \emph{connected component},
	when applied to an oriented graph, implicitly apply to its underlying
	graph.

	\subsection*{Notation for binary relations}
	
	Let $ S $ be a set, and let $ \Rl $ be a binary relation on $
        S $. We write $ x \Rl y $ for $ (x,y) \in \Rl $, and
        $ x \notRl y $ for $ (x,y) \notin \Rl $. For an element
        $ s\in S $, we denote by $ [s \Rl] $ the set
        $ \{t \in S : s \Rl t \} $.

        The relation $ \Rl $ is \emph{asymmetric} if for all
        $x,y \in S $, $ x \Rl y $ implies $ y \notRl x $, and it is
        \emph{transitive} if for all $x,y,z \in S$, $ x \Rl y $ and
        $ y \Rl z$  implies $x \Rl z $. The relation $ \Rl $
        is a \emph{strict partial order} if it is asymmetric and transitive.
%
%

	A \emph{directed cycle} in $ \Rl $ is a set of elements
        $ x_1, x_2, \dots, x_n $, with $ n \in \mathbb{N} $, such that
        $ x_1\Rl x_2$, $ x_2 \Rl x_3 $, $\dots$, $ x_n \Rl x_1 $. Note
        that when we deal with relations, we allow cycles on one or
        two elements. So, strict partial orders do not have
        directed cycles. In fact, a relation $\Rl $ has no directed
        cycles if and only if its transitive closure is a strict
        partial order.
	
	An element $ s \in S $ is said to be a \emph{minimal element}
        with respect to $ \Rl $ if there exists no element
        $ t \in S \sm \{s\} $ such that $ t \Rl s $. Notice that if a
        relation $ \Rl $ on a finite set $ S $ has no directed cycle,
        then $ S $ necessarily has a minimal element with respect to
        $ \Rl $.

        \section{Derived graphs}
        \label{section:Derived-graphs}
	
	In this section, we introduce the class of \emph{derived
          graphs} and study some of their basic properties.  A
        \emph{Burling tree} is a 4-tuple
        $(T,r, \lastBorn, \choosePath)$ in which:
	
	\begin{enumerate}[label=(\roman*)] 
		\item $ T $ is a rooted tree and $ r $ is its root,
		\item $\lastBorn$ is a function associating to each vertex $v$ of $T$
		which is not a leaf, one child of $v$ which is called the
		\emph{last-born} of $v$,
		\item $\choosePath$ is a function defined on the vertices of $ T $. If
		$ v $ is a non-last-born vertex in $ T $ other than the root, then
		$ \choosePath $ associates to $ v $ the vertex-set of a (possibly empty)
		branch in $T$ starting at the last-born of $p(v)$. If $v$ is a
		last-born or the root of $T$, then we define
		$ \choosePath(v) = \varnothing $. We call $ \choosePath $ the \textit{choose
			function} of $ T $.
	\end{enumerate}
	
	By abuse of notation, we often use $T$ to denote the 4-tuple.

	The oriented graph $G$ \emph{fully derived} from the Burling tree $T$
	is the oriented graph whose vertex-set is $V(T)$ and
	$uv \in A(G) $ if and only if $v$ is a vertex in $\choosePath(u)$.  A
	non-oriented graph $G$ is \emph{fully derived} from $T$ if it is the
	underlying graph of the oriented graph fully derived from $T$.
	
	A graph (resp.\ oriented graph) $G$ is \emph{derived} from a Burling
	tree $T$ if it is an induced subgraph of a graph (resp.\ oriented
	graph) fully derived from $T$. The oriented or non-oriented graph
	$G$ is called a \emph{derived graph} if there exists a Burling
	tree $T$ such that $G$ is derived from $T$.
	
	Observe that if the root of $T$ is in $V(G)$, then it is an isolated
	vertex of~$G$.  Observe that a last-born vertex of~$T$ that is in $G$
	is a sink of $G$. 
	
	Let us give some examples. In all figures in this paper, the tree~$T$ is represented with black edges while
        the arcs of $G$ are represented in red. The last-born of a
        vertex of $T$ is presented as its rightmost child. Moreover, \emph{shadow
          vertices}, the vertices of $T$ that are not in $G$,
        are represented in white.
	
	\begin{figure}
          \begin{center}
            \includegraphics[width=8cm]{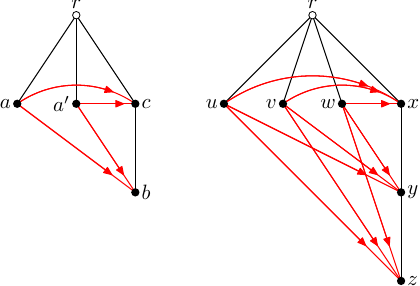}
            \caption{Complete bipartite graphs seen as derived
              graphs\label{f:square}}
          \end{center}
	\end{figure}
	
	On the first graph represented in Figure~\ref{f:square},
        $\choosePath(x) = \choosePath(y) = \{z,w\} $.  It shows that
        at least one orientation of $C_4$ is a derived graph,
        so that $C_4$ is a derived graph.  The second graph shows that
        $K_{3, 3}$ is a derived graph, and it is easy to generalize
        this construction to $K_{n, m}$ for all integers
        $n, m \geq 1$. In both graphs, the vertex $ r $ of $ T $ is
        not a vertex of $ G $. Figure \ref{f:c6} is a presentation of
        $ C_6 $ as a derived graph. Notice that, in this presentation,
        $ v $ is a shadow vertex.
	
	\begin{figure}
		\begin{center}
			\includegraphics[width=5cm]{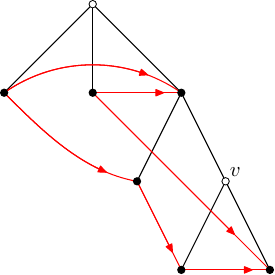}
			\caption{Cycle of length 6 seen as a derived graph\label{f:c6}}
		\end{center}
	\end{figure}
	
%
	Notice that if a graph $G$ is derived from $T$, the branches of $T$,
	restricted to the vertices of $G$, are stable sets of $G$. In
	particular, no edge of $T$ is an edge of $G$.

	Let $G$ be an oriented graph derived from a Burling tree $T$. A
	vertex~$v$ in~$G$ is a \emph{top-left} vertex if its distance in $T$
	to the root of $T$ is minimum among all vertices of $G$, and one of
	the followings holds:
	
	\begin{enumerate}[label=(\roman*)]
		\item $v$ is not a last-born,
		\item $ v$ is a last-born and every vertex of $ G $ whose distance in $ T $ to the root is minimum is also a last-born. 
	\end{enumerate}
	
%
%
	There might be more than one top-left 
	vertex in
	a graph. For example, in the first graph of Figure \ref{f:square},
	both vertices $ x $ and $ y $ are top-left vertices.

	\begin{lemma}
		\label{lem:top-left-exists}
		Every non-empty oriented graph $G$ derived from a Burling tree
		$(T, r, \lastBorn, \choosePath)$ contains at least one top-left
		vertex and every such vertex is a source
		of $G$. Moreover, the neighborhood of a top-left
		vertex is a stable set.
	\end{lemma}
	
	\begin{proof}
		By the definition of top-left vertex, it exists in $ G $. Let $v$ be
		a top-left vertex of $G$. Suppose for the sake of contradiction that
		$uv \in A(G)$ for some vertex $u \in V(G)$. Thus $v$ is a vertex in
		$\choosePath(u)$. Denote by $ d(x) $ the distance in $ T $ of a
		vertex $ x $ to $r $. The fact that $ v \in \choosePath(u) $ means
		that $v$ is a descendant of a brother of $u$, and therefore
		$ d(v) \geq d(u) $. Since $v$ is a vertex that minimizes the
		distance to the root, we must have $ d(v)=d(u) $, and in particular
		$ p(v) = p(u) $. Notice that $ u $ and $ v $ cannot both be last-born. 
		On the other hand, $v$ is a last-born because
		$u$ cannot be connected to one of its non-last-born brothers. This
		contradicts the definition of a top-left vertex. So $N(v) =
		N^+(v)$. It follows that $N(v)$ is included in a branch of $T$, and
		is therefore a stable set.
%
	\end{proof}
	
	\begin{lemma}
		\label{lem:DG-no-cycle-no-triangle}
		An oriented derived graph has no directed cycles and its underlying graph has no triangles.
	\end{lemma}
	
	\begin{proof}
		Adding a source whose neighborhood is a stable set to an oriented
		graph with no directed cycle and no triangle does not create a
		triangle or a directed cycle. Since every induced subgraph of a
		derived graph is a derived graph, the statement follows from Lemma
		\ref{lem:top-left-exists} by a trivial induction.
	\end{proof}

	Suppose that $(T, r, \lastBorn, \choosePath)$ is a Burling tree, $u$
	is a non-leaf vertex of $T$ and $v$ is its last-born.  Suppose that
	$b$ is a non-last-born child of $u$. Consider the tree $T’$ obtained
	from $T$ by removing the edges $uv$ and $ ub $, and adding a vertex $w$ adjacent to
	$u$, $v$ and $b$.  Define $\lastBorn'(u) = w$, $\lastBorn'(w) = v$ and
	$\lastBorn'(z)=\lastBorn(z)$ for all non-leaf vertices $z$ of
	$T\sm \{u\}$.  Define $\choosePath'(z) = \choosePath(z)\cup \{w\}$ for
	every vertex $z\in V(T)\setminus\{b\}$ such that $v \in \choosePath(z)$ or
	$b\in \choosePath(z)$, define $\choosePath'(w)=\varnothing$,
        and define $\choosePath'(z)=\choosePath(z)$
	otherwise. See Figure~\ref{f:kill3}.
	
	\begin{definition}
		\label{d:subdivB}
		The Burling tree $(T', r', \lastBorn', \choosePath’)$ defined above
		is said to be obtained from $(T, r, \lastBorn, \choosePath)$ by
		sliding $b$ into $uv$ (note that the definition requires that $v$ is
		a last-born).
	\end{definition}

	\begin{figure}[t]
		\begin{center}
			\includegraphics[width=10cm]{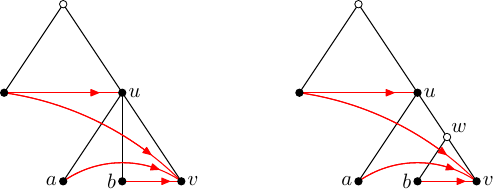}
			\caption{Sliding $b$ into $uv$}
			\label{f:kill3}
		\end{center}
	\end{figure}

	\begin{lemma}
		\label{l:slide}
		If $(T', r', \lastBorn', \choosePath’)$ is obtained from
		$(T, r, \lastBorn, \choosePath)$ by sliding a vertex into an edge,
		then any oriented graph derived from
		$(T, r, \lastBorn, \choosePath)$ can be derived from
		$(T', r', \lastBorn', \choosePath’)$.
	\end{lemma}
	
	\begin{proof}
		Let $G$ be derived from $T$. The statement follows directly from the
		fact that the function $\choosePath$ is the restriction of
		$\choosePath’$ to $V(G)$.
	\end{proof}

	The next lemma shows that all derived graphs can be derived from
	Burling trees with specific properties. This will reduce the technical difficulty of
 some proofs.

	\begin{lemma}
		\label{l:chooseT}
		Every oriented derived graph $G$ can be derived from a Burling tree
		$ (T, r, \lastBorn, \choosePath) $ such that:
		\begin{enumerate}[label=(\roman*)]
			\item $ r $ is not in $V(G)$, \label{item:1-in-lem-chooseT}
			\item every non-leaf vertex in $T$ has exactly two children,
			\item no last-born of $T$ is in $V(G)$.
		\end{enumerate}
	\end{lemma}
	
	\begin{proof}
		We apply a series of transformations on
		$ (T, r, \lastBorn, \choosePath) $ until the conclusion is
		satisfied.
		
		\textit{First transformation:} If $ r \in V(G) $,
                build a tree $T'$ by adding to $T$ a new vertex $r'$
                adjacent to $r$.  Define $\lastBorn'(r') = r $ and
                $\lastBorn'(v)=\lastBorn(v)$ for all vertices $v$ of
                $T$. Moreover set $ \choosePath'(r') = \varnothing $,
                and do not change the choose function on the rest of
                the vertices. Notice that $ r $ is an isolated vertex
                in $ G $, thus $G$ can be derived from
                $(T', r', \lastBorn', \choosePath')$.

		\textit{Second transformation:} Suppose that $ u $ is a non-leaf
		vertex of $ T $ which has only one child. Build a tree $ T' $ by
		adding a new child $ v $ to $u $ and define
		$ \choosePath'(v) = \varnothing $. Notice that $ v $ is a leaf, so
		it does not have a last-born in $ T' $. The graph $ G $ is also
		derived from $ T ' $. Apply this process until that every non-leaf
		vertex in $ G $ has at least two children.

		\textit{Third transformation:} Suppose that $u$ is a vertex in $T$
		with at least three children, let $v$ be the last-born of $u$ and
		$a, b$ be two distinct children of $u$ other than $ v $. We define a
		Burling tree $T'$ by sliding $b$ into the edge $uv$, and observe
		that the degree of $u$ in $T'$ is smaller than in $T$. And by
		Lemma~\ref{l:slide}, $G$ can be derived from $T'$. We apply the
		transformation until all vertices have at most two children.
		
		Notice that during this process we decrease the number of children
		of $ u $, the new vertex $w$ has two children, and we do not
		increase the number of children of any other vertex. Hence the
		process terminates if we apply the transformation until conclusion
		(ii) of the lemma is satisfied.
		
		Moreover, notice that in applying the third transformation on a vertex
		$ u $, we do not decrease the number of children of any vertex other
		than $ u $, and once again the new vertex that we create has two
		children. Thus, after the third transformation, we do not undo the
		effect of the second transformation.
		
		Notice that after the second and the third transformations, Property
		\ref{item:1-in-lem-chooseT} of the lemma remains satisfied.
		
		\textit{Fourth transformation:} If $v$ is a last-born of $T$ that is
		in $V(G)$, then let $u$ be the parent of $v$. Observe that
		$\choosePath(v)=\varnothing$.  We build a tree $T'$ by removing the
		edge $uv$, adding a new vertex $w$ adjacent to $u$ and $v$, and a
		new vertex $x$ adjacent to $w$. Define $\lastBorn'(u) = w$,
		$\lastBorn'(w) = x$ and $\lastBorn'(y)=\lastBorn(y)$ for all
		non-leaf vertices $y$ of $T\sm u$.  Define
		$\choosePath'(y) = \choosePath(y)\cup \{w\}$ for every vertex
		$y\in V(T)$ such that $v \in \choosePath(y)$ and
		$\choosePath'(y)=\choosePath(y)$ otherwise.  We see that $G$ can be
		derived from $(T', r', \lastBorn', c')$, and $ v $ is not a
		last-born in $ T' $, so we have reduced the number of last-borns of
		the Burling tree in $ V(G) $. Apply this transformation until there
		is no last-born of the Burling tree in $ V(G)$. See
		Figure~\ref{f:killlb}.
		\begin{figure}[t]
			\begin{center}
				\includegraphics[width=10cm]{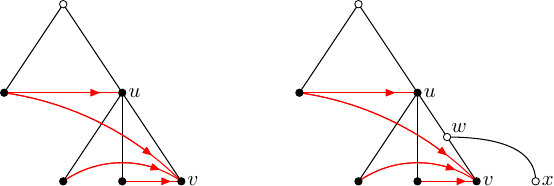}
				\caption{Turning $v$ into a non-last-born}
				\label{f:killlb}
			\end{center}
		\end{figure}  
		
		Finally, notice that this transformation does not cancel the effect
		of the previous ones.  This completes the proof of the lemma.
	\end{proof}

	\section{Equality of Burling graphs and derived graphs}
        \label{sec:BequalD}
	
        In this section, we recall the classical definition of the
        Burling sequence and prove that derived graphs and Burling graphs form the same class. 
		
		
	\subsection*{Burling graphs}
	
	There are different equivalent approaches to define
        Burling graphs. See~\cite{Burling65} for the first definition
        by Burling, or \cite{Pawlik2012Jctb} for a
        second definition. The definition that we
        use here is the one from~\cite{Chalopin2014} (see Appendix~B
        of~\cite{Chalopin2014}).
	
	\begin{definition}
          \label{def:BG}
          Let $(G, \mathcal S)$ be a pair where $G$ is a graph and
          $\mathcal S$ is a set of stables sets of~$G$. We define a
          function $ \nextB $ associating to a pair $(G, \mathcal S)$
          another pair~$ (G', \mathcal S') $ as follows:
          \begin{enumerate}[label=(\roman*)]
          \item Take a copy of $G$. 
          \item For each stable set $S \in \mathcal S$, take a new
            copy of $G$ and denote it by $G_{\mathcal{S}}$. Note that
            the same set of stable sets as $ \mathcal S $ exists in
            $G_{\mathcal{S}}$. Denote it by~$ \mathcal S_S $.			
          \item For each $ S \in \mathcal S $ and
            $ Q \in \mathcal{S}_S $, add a new vertex $v_{S,Q}$
            adjacent to all vertices in $ Q $.
          \item Denote by $G'$ the obtained
            graph	
          \item Consider every stable set of the form $ S \cup Q $ and
            $ S \cup \{v_{S,Q} \}$ where $ S \in \mathcal{S} $ and
            $ Q \in \mathcal S_S $. Call $\mathcal S'$ the set of all
            these stable sets.
          \end{enumerate}
		
          The pair $(G', \mathcal S')$ is defined to be
          $\nextB(G, \mathcal
          S)$.
	\end{definition}
	
	Starting with $(G_1, S_1)$ where $G_1=K_1$ and
        $\mathcal S_1 = \{V(G_1)\}$ and applying the function $\nextB$
        iteratively, we define a sequence $ (G_k)_{k \geq 1} $ in
        which
        $(G_{k+1}, \mathcal S_{k+1}) = \nextB(G_k, \mathcal
        S_k)$. This sequence is called \emph{the Burling sequence}. In
        Figure~\ref{f:example-Burling-sequence}, the first three
        graphs in this sequence are represented. The edges of the
        graphs are represented in red and the stable sets are
        represented by dashed curves.
	
	\begin{figure}
          \begin{center}
            \vspace*{-2.5cm}
            \includegraphics[width=14cm]{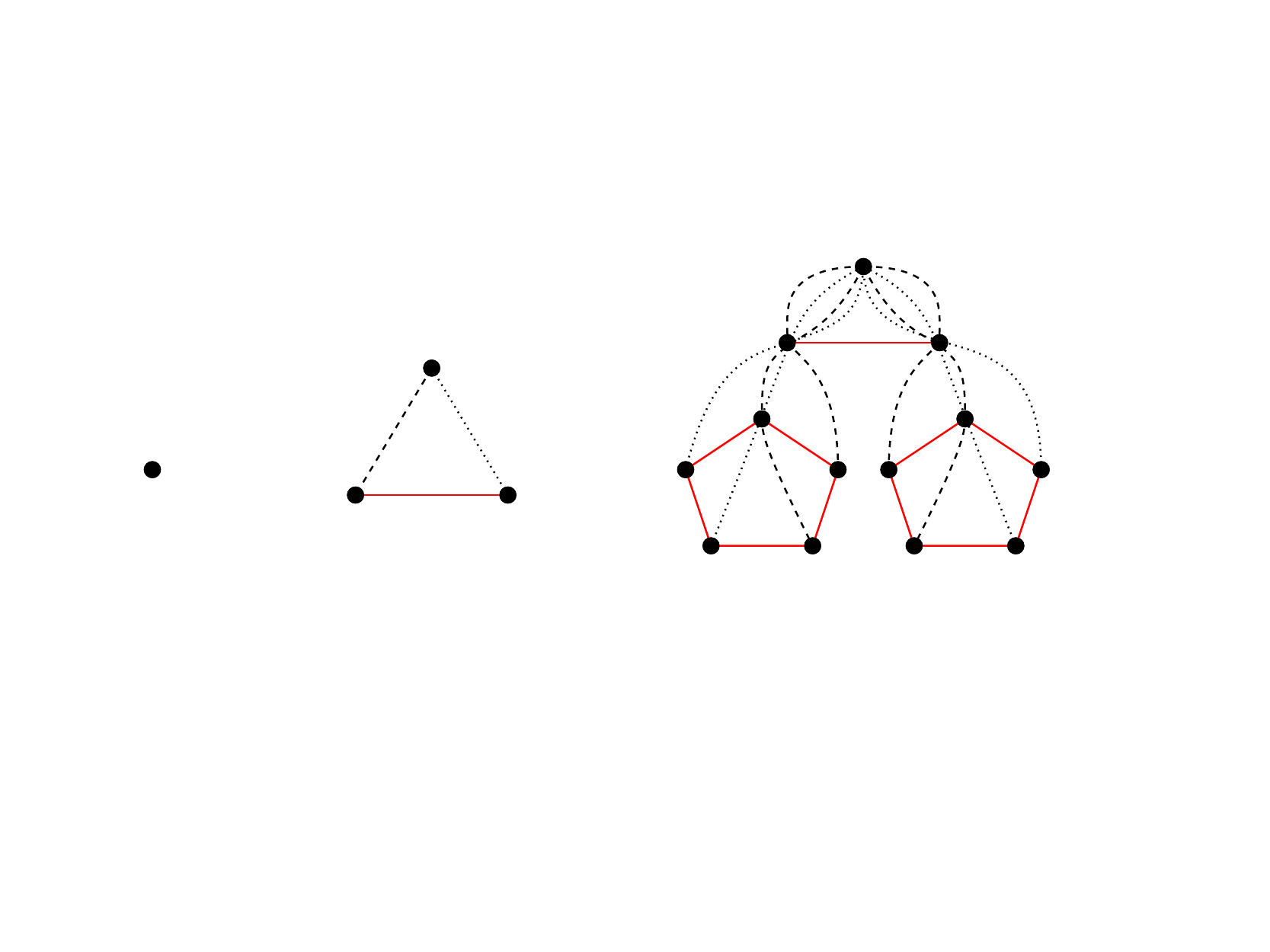}
            \vspace*{-4.5cm}
            \caption{The first three graphs in the Burling sequence}
            \label{f:example-Burling-sequence}
          \end{center}
	\end{figure}

	Notice that a copy of the first graph $ G_1 $, which is a
        single vertex, is present in all the graphs of the sequence,
        and it is an isolated vertex of them. 
	
	The class of \emph{Burling graphs} is the class consisting of
        all graphs in the Burling sequence and their induced
        subgraphs.
	
	Burling proved that the graphs of the Burling sequence have
        unbounded chromatic number. (See Theorem~1
        of~\cite{Pawlik2012Jctb}.) For the sake of completeness, we
        include the sketch of the proof here. Here, a \emph{coloring}
        of a graph is a function that assigns to each vertex a
        color, in such a way that adjacent vertices receive different
        colors. By induction, we prove the following statement:
	
	\textit{In every coloring of the vertices of $G_k$, one of the
          stable sets in the family $\mathcal S_k$ receives at least
          $k$ colors.}
	
	This is obvious for $k=1$. Suppose the statement holds for
        some fixed~$k$. Consider a coloring of $ G_{k+1} $. By the
        induction hypothesis, in the first copy of $ G_k $ in
        $ G_{k+1} $, there exists a stable set $ S \in \mathcal S_k $
        which receives at least $k $ colors. Again, by the induction
        hypothesis, in $ G_S $, the copy of $ G_k $ associated to
        $ S $, there exists a stable set $ Q \in \mathcal S_S $
        receiving $k $ colors. Now either the $k $ colors of $ S $ are
        the same as the $ k $ colors of $ Q $, in which case
        $ v_{S,Q} $ has a new color, and therefore
        $ S \cup \{v_{S,Q}\} \in \mathcal S_{k+1} $ receives $ k+1 $
        different colors, or the colors in $ S $ and $ Q $ are
        different, in which case $ S \cup Q \in \mathcal S_{k+1} $
        receives $ k+1 $ different colors. This completes the proof.

	\subsection*{Tree sequence}
	
	Recall that a \emph{principal branch} of a Burling tree
        $ (T, r, \lastBorn, \choosePath) $ is any branch starting in
        its root $ r $ and ending in one of its leaves. The
        \emph{principal set of $(T, r, \lastBorn, \choosePath)$} is
        the set of all vertex-sets of the principal branches of
        $T$. We denote the principal set of $ T $ by
        $ \mathcal{P} (T)$. Notice that there is a one-to-one
        correspondence between $ \mathcal{P} (T)$ and $L(T)$, the set of
        leaves of $ T $.
	
	If a graph $ G $ is derived from a Burling tree $ T $, then
        the restriction of each principal branch of $ T $ to the
        vertices of $ G $, forms a stable set in $ G $. In particular,
        $ \mathcal P(T) $, restricted to $ V(G) $, is a set of stable
        sets of $ G $.
	
	In this section, we define a sequence $ (T_k)_{k \geq 1}$ of
        Burling trees and we prove that the sequence
        $ (T_k, \mathcal P_k)_{k \geq 1} $ of Burling trees and their
        principle sets is in correspondence to the sequence
        $(G_k, \mathcal{S}_k)_{k \geq 1}$ of Burling graphs. More
        precisely, we will show that the $k$-th Burling graph $ G_k $
        is isomorphic to the graph fully derived from $ T_k $, and
        $ S_k $ is the same as $ \mathcal P_k = \mathcal P(T_k) $.
	
	To define the mentioned sequence, we first define a
        function $ \nextT $ on Burling trees.
	
	\begin{definition} \label{def:BurlingTreeNext} Let
          $(T, r, \lastBorn, \choosePath)$ be a Burling tree, and let
          $ \mathcal S $ denote its principal set. We build a Burling
          structure $ (T', r', \lastBorn', \choosePath') $ with
          principal set $ \mathcal S' $ as follows:
		\begin{enumerate}[label=(\roman*)]
                \item Take a copy of
                  $ (T, r, \lastBorn, \choosePath) $.
                \item For each principal branch $ P \in \mathcal{S} $
                  ending in the leaf $ l $, pend a leaf $ l_P $ to
                  $ l $, and define $ \lastBorn(l) = l_P $. Then put a
                  copy $ (T, r, \lastBorn, \choosePath)_P $ on
                  $ l_P $, identifying its root with $ l_P$. Denote
                  the principal set of
                  $ (T, r, \lastBorn, \choosePath)_P $ by
                  $ \mathcal S_P $.
                \item For each copy
                  $ (T, r, \lastBorn, \choosePath)_P $, corresponding
                  to a leaf $ l \in P $, for each
                  $ Q \in \mathcal{S}_P $, add a new leaf $ l_{P,Q} $
                  to $ l $.
                \item to obtain $ c' $, first extend the function
                  $ \choosePath $ naturally to the copies of
                  $ (T, r, \lastBorn, \choosePath) $, and then also
                  define $ \choosePath'(l_{P,Q}) = Q $ for
                  $ P \in \mathcal S $ and $ Q \in \mathcal{S}_P $.
                \item Notice that the result is a Burling tree
                  $ (T', r', \lastBorn', \choosePath') $.
                \item Observe that the principal branches of $ T' $
                  are of the form $ P \cup Q $ or
                  $ P \cup \{l_{P,Q}\} $ for $ P \in \mathcal S $ and
                  $ Q \in \mathcal{S}_P $. Thus
                  $ \mathcal{S}' = \{ P \cup Q, P \cup \{l_{P,Q}\} : P
                  \in \mathcal S, Q \in \mathcal{S}_P \} $.
		\end{enumerate}
		We denote $ (T', r', \lastBorn', \choosePath') $ by
                $ \nextT(T, r, \lastBorn, \choosePath) $. By abuse of
                notation, we may write~$ T'=\nextT(T) $.
	\end{definition}
	
	Starting from $ T_1 $, the one vertex Burling tree, and
        applying the $\nextT$ function iteratively, we reach a
        sequence $ (T_k, r_k, \lastBorn_k, \choosePath_k)_{k \geq 1} $
        of Burling trees that we call \textit{the tree sequence}.
	
	In the rest of this section whenever we use the notation
        $ (G_k, \mathcal S_k) $, we mean the $k$-th graph in the
        Burling sequence and its set of stable sets. Similarly, when
        we write $ (T_k, r_k, \lastBorn_k, \choosePath_k) $, or by
        abuse of notation $ T_k $, we mean the $k$-th Burling tree in
        the tree sequence.

	The next two lemmas are about some properties of the sequence
        $ (T_k)_{k \geq 1} $.
	
	\begin{lemma} \label{lem:Tk-two-children} Let $ v $ be a
          vertex in $ T_k $. If $ v $ is not a leaf, then it has at
          least two children in $ T_k $.
        \end{lemma}
	
	\begin{proof}
          We prove the lemma by induction on $ k $. For $ k=1 $, there
          is nothing to prove. Suppose that the statements are true
          for $ T_k $ where $ k \geq 1 $.
		 
          Let $ v $ be a vertex in $ T_{k+1} = \nextT (T_k)$ which is
          not a leaf. The vertex $ v $ appears in one of the copies of
          $ T_k $, and because it is not a leaf, either it is a
          non-leaf vertex of a copy of $ T_k $, and thus it has at
          least 2 children by the induction hypothesis, or it is a
          leaf of the main copy of $ T_k $ in $ T_{k+1} $. But notice
          that as a leaf of the main copy of $ T_k $, in Step~(ii) of
          Definition~\ref{def:BurlingTreeNext}, it receives a child,
          and in Step~(iii) it receives at least one more child. So
          $v $ has at least 2 children in $ T_{k+1}$.
	\end{proof}
	
	\begin{lemma} \label{lem:Tk-nonempty-c} If $ v $ is a
          non-last-born vertex in $T_k $ which is not the root, then
          $ c_k(v) \neq \varnothing $. In particular, the last-born
          brother of $ v $ is in $ c_k(v) $.
	\end{lemma}
	
	\begin{proof}
          We prove the lemma by induction on $ k $. For $ k=1 $, there
          is nothing to prove. Suppose that the statements are true
          for $ T_k $ where $ k \geq 1 $, and suppose that $ v $ is a
          non-last-born vertex in $ T_{k+1} $ other than its
          root. There are two possibilities:
		
          First, $ v $ is a non-last-born vertex in one of the copies
          of $ T_k $ (either the main copy, or a copy corresponding to
          a principal branch). In this case, the result follows from
          the induction hypothesis.
		
          Second, $ v $ is a vertex of the form $ l_{P,Q} $ as in
          Step~(iii) of Definition~\ref{def:BurlingTreeNext}. Then in
          Step~(iv) we define $ c_{k+1}(v) $ to be $ Q $ which is not
          empty.
	\end{proof}

	\subsection*{Equality of Burling graphs and derived graphs}
	
	We are now ready to prove the equality of Burling
        graphs and derived graphs. 
	
	
	\begin{lemma} \label{lem:Gk-is-derived-from-Tk}
		For every $ k \geq 1$, $ G_k $ is fully derived from $ T_k $, and $ \mathcal S_k $ is $\mathcal P_k = \mathcal{P}(T_k)$.
	\end{lemma}
	
	\begin{proof}
		We prove the lemma by induction on $ k $. If $ k=1 $, the statement holds. Suppose that $ G_k $ is fully derived from $ T_k $ and $\mathcal S_k $ is equal to $\mathcal{P}_k$. 
		
		To build $ T_{k+1} $, to every leaf $ l $ of $ T_k $, we add a new leaf and we pend a copy of $ T_k $ to this new leaf. Since every leaf in $ T_k $ identifies exactly one of the principal branches, or by the induction hypothesis, one stable set in $ \mathcal S_k $, this step is equivalent to step~(ii) in Definition~\ref{def:BG}. Then for each copy $ (T_k)_P $ of $ T_k $, we add $|\mathcal P_k| = |\mathcal S_k|$ new leaves to the leaf corresponding of the principal branch $ P $. For a new vertex $ l_{P,Q} $ corresponding to the branch $ Q \in \mathcal{P}(T_k)_P $, we define the choose-function to be $ Q  \in (T_k)_P $ which assures that in the graph fully derived from $ T $, this vertex is complete to $ Q $. Thus these new vertices $ l_{P,Q} $ are the vertices $ v_{P,Q} $ that we add in step~(iii) of Definition~\ref{def:BG}, and $ G_{k+1} $ is the graph fully derived from $ T_{k+1} $. 
		
		Finally, we notice that the vertex sets of the principal branches of $ T_{k+1} $ are exactly sets of the form $ P \cup Q $ and $ P \cup \{l_{P,Q}\} $ for $ P \in \mathcal \mathcal P_k = \mathcal S_k$ and $ Q \in (\mathcal P_k)_P =(\mathcal{S}_k)_P$. Thus $ \mathcal S_{k+1} = \mathcal P_{k+1} $. 
	\end{proof}

	\begin{figure}[t]
		\begin{center}
			\includegraphics[width=12cm]{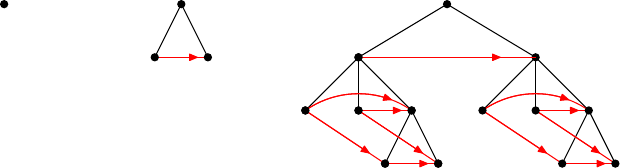}
			\caption{The first three graphs of the Burling sequence seen as fully derived graphs} \label{f:Burling-sequence-as-derived}
		\end{center}
	\end{figure}
	
	Figure~\ref{f:Burling-sequence-as-derived} shows some orientations of the first three graphs of the Burling sequence as fully derived graphs.

	Now we define the notion of \textit{extension} for 
	Burling trees, which is, as we will see formally in Lemma \ref{lem:extension-induced-subgraph}, closely related 
	to the notion of induced subgraph in fully derived graphs. 
	
	\begin{definition} \label{def:monotonic-embedding}
		Let $ (T, r, \lastBorn, \choosePath) $ and $ (T', r', \lastBorn', \choosePath') $ be two Burling trees. We say that $ T' $ is an \emph{extension} of $ T $ if there exists an injection $ \varphi $ from $V(T) $ to $ V(T') $ with the following properties:
		\begin{enumerate}[label=(\roman*)]
			\item $\varphi(r) = r'$,
			\item $\varphi$ preserves ancestors, i.e.\ if $ u $ is an ancestor of $ v $ in $ T $, then $ \varphi(u) $ is an ancestor of $ \varphi(v) $ in $ T' $,
			\item $\varphi$ preserves the last-born vertices, i.e. if $ v \in V(T) $ is a last-born in $ T $, then  $\varphi(v) $ is a last-born in $ T'$.
			\item $\varphi $ preserves the choose-path function on $ T $, i.e.\ for every vertex $ v \in T $, $ \varphi(\choosePath(v)) = \choosePath'(\varphi(v)) \cap \varphi(V(T)) $. \label{item:iii-monoton}
		\end{enumerate}
	\end{definition}

	\begin{lemma} \label{lem:extension-induced-subgraph}
		Let $ G $ and $ G' $ be two oriented graphs fully derived from $ T $ and $ T' $ respectively. If $ T' $ is an extension of $ T $, then $ G $ is an induced subgraph of~$ G$. 
	\end{lemma}
	
	\begin{proof}
		Let $ \varphi $ be the injection from $ V(T) $ to $ V(T')$. Since $ G $ and $ G' $ are fully derived from $ T $ and $ T' $, $ V(G) = V(T) $ and $ V(G') = V(T') $. Thus $ \varphi $ can be seen as an injection from $ V(G) $ to $ V(G') $. By property \ref{item:iii-monoton} in Definition \ref{def:monotonic-embedding}, $v \in \choosePath(u) $ if and only if $ \varphi(v) \in \choosePath'(\varphi(u)) $. In other words, $ uv \in A(G) $ if and only if $ \varphi(u) \varphi (v) \in A(G')$. Thus $ G $ is an induced subgraph of $ G' $.      
	\end{proof}

	Next lemma shows that the tree sequence $ (T_k)_{k \geq 1} $ contains all the Burling trees in the extension sense.
	
	\begin{lemma} \label{lem:Tk-extension-T}
		If $ (T, r, \lastBorn, \choosePath) $ is a Burling tree such that every non-leaf vertex has exactly two children, then there is an integer $ i \geq 1 $ such that $ T_i $ is an extension of~$ T  $.
	\end{lemma}
	
	\begin{proof}
		We prove the lemma by induction on the number of vertices of $ T $.
		
		
		For the induction step, the smallest possible $ T $ is a tree on three vertices: the root $ r $, the last-born of the root $v $, and the other child of the root $ u$. If $ \choosePath(u) = \{v\} $, then $ T_2 $ is an extension of $ T $. If $ \choosePath(u) = \varnothing $, then $ T_3 $ is an extension of $ T $ as shown in Figure~\ref{pic:T3-extensio-of-T}.
		
		\begin{figure} 
			\vspace*{-2cm}
			\centering
			\includegraphics[width=10cm]{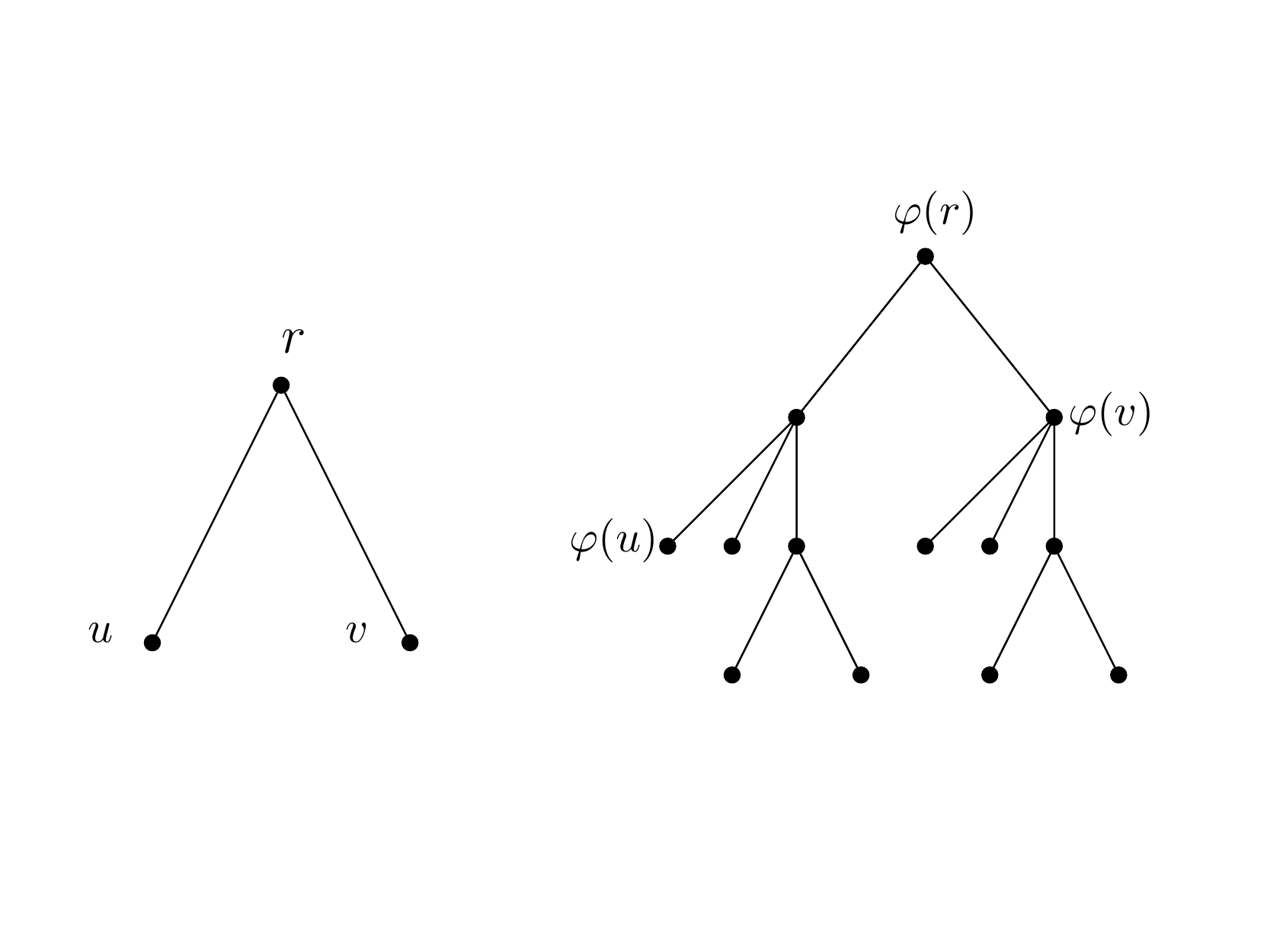}
			\vspace*{-2cm}
			\caption{When $ \choosePath(u) = \varnothing $, $T_3$ (right) is an extension of $ T $ (left)} \label{pic:T3-extensio-of-T}
		\end{figure}
		
		Suppose that the lemma is true for every Burling tree on at most $ n $ vertices. Suppose that $ T $ on $ n > 1 $ vertices is given.
		
		Consider the set of all the vertices of $ T  $ which have the maximum distance to $ r $. Because every non-leaf vertex in $ T $ has two children, there is a non-last-born vertex $ x $ in this set. Notice that $ x $ has no children. Denote by $ p $ the parent of $ x $ and by $ y $ the last-born of $ p $. Notice that $ y $ also has the maximum distance to the root, and thus both $ x $ and $ y $ are leaves of $ T $.
		
		Consider the tree $ (T', r, \lastBorn', \choosePath') $, obtained from $ T $ by removing the two leaves $ x $ and $ y $, and restricting the functions $ \lastBorn $ and $ \choosePath $. By induction hypothesis, there exist $ k $ such that $ T_k $ is an extension of $ T $. Let $ \varphi $ be the injection from $V(T)$ to $ V(T_k) $. In the rest of the proof, we will define $\varphi$ on $ x $ and $ y $ in order to extend $\varphi$ to $V(T)$, in a way that all the four properties of Definition \ref{def:monotonic-embedding} remain satisfied. 
		
		Now there are two possible cases. 
		
		\textit{Case 1:} $ y \in \choosePath(x) $. 
		
		If $ \varphi(p) $ is not a leaf of $ T_k $, then define $\varphi(x)$ to be a non-last-born child of $ \varphi(p) $, which exists by lemma \ref{lem:Tk-two-children}, and define $ \varphi(y) $ to be the last-born of $ \varphi(p) $. By Lemma \ref{lem:Tk-nonempty-c}, $\varphi(y)$ is in $ c_k(\varphi(x)) $. Notice that this extension of $\phi $ has all the properties of Definition \ref{def:monotonic-embedding}. Properties~(i) to~(iii) are easy to verify, and for Property~(iv), notice that no descendant of $ \varphi(y) $ is in the image of~$\varphi$, thus $ \varphi(\choosePath(x)) = \varphi(\{y\}) = \{\lastBorn_k (p)\} = c_k(\varphi(y)) \cap \im(\varphi) $.
		
		If $ \varphi(p) $ is a leaf of $ T_k $, then consider $ T_{k+1} $. In building $ T_{k+1} $, every leaf of the first copy of $ T_k $, including $ \varphi(p) $, will receive a last-born and at least one non-last-born child. Define again $\varphi(x)$ to be a non-last-born child of $ \varphi(p) $ and $ \varphi(y) $ to be the last-born of $ \varphi(p) $. See Figure \ref{pic:Tk-extension-T-case1}. Notice that again $ \varphi $ has all the required properties. So $ T_{k+1} $ is an extension of $ T $. 
		
		\begin{figure}
			\centering
			\vspace*{-3cm}
			\includegraphics[width=12cm]{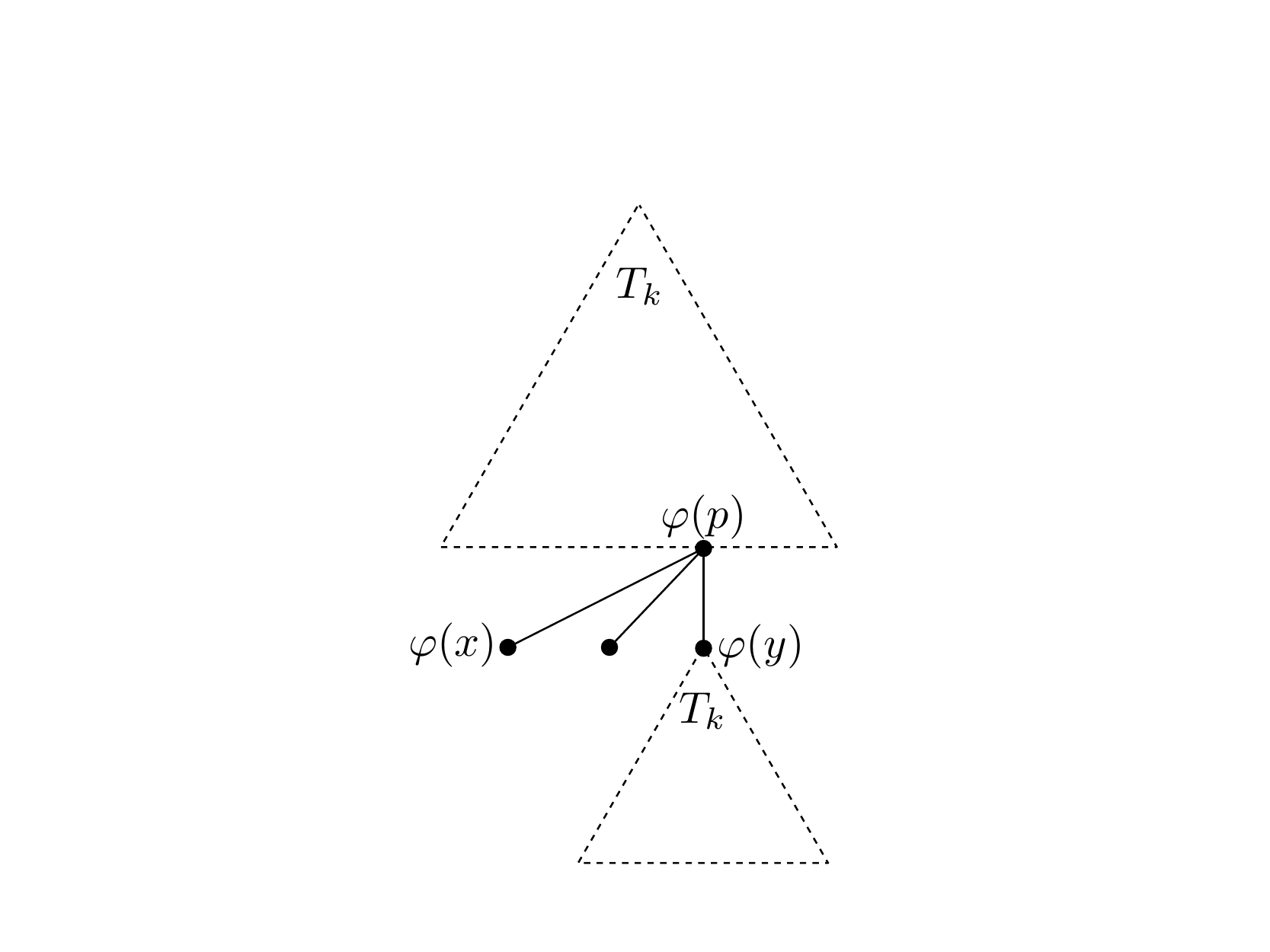} 
			\vspace*{-1cm}
			\caption{Case 1 of the proof of Lemma \ref{lem:Tk-extension-T}} \label{pic:Tk-extension-T-case1}
		\end{figure}
		
		\textit{Case 2:} $ y \notin \choosePath(x) $.
		
		If $\varphi(p)$ is not a leaf of $ T_k $, by \ref{lem:Tk-two-children} it has at least two children. Choose two paths starting at two different children of $ \varphi(p) $ and ending at two different leafs $l $ and $ \lastBorn' $ of $ T_k $. In $ T_{k+1} $, consider $ l $ and $ \lastBorn' $ in the first copy of $ T_k $. Define $ \varphi(x) $ to be some non-last-born of $ l$ in $ T_{k+1} $ and $ \varphi(y) $ to be the last-born of $ \lastBorn' $ in $ T_{k+1}$. See Figure \ref{pic:Tk-extension-T-case2-thetwo}, left. Notice that $ l \neq \lastBorn' $, thus $ \varphi(y) \notin \varphi(x) $. The new function $\varphi$ has all the required properties. Hence $ T_{k+1} $ is an extension of~$ T$.
	
		If $\varphi(p)$ is a leaf of $ T_k $, then consider $ T_{k+1} $. In $ T_{k+1} $, the vertex $ \varphi(p) $ in the main copy of $ T_k $ has a last-born $ l $ and at least one non-last-born. Choose any non-last-born child of $ \varphi (p) $ and denote it by $ n $. Notice that $ n $ is a leaf of $ T_{k+1} $. Thus in $ T_{k+2} $, this vertex will have a some children, including at least one non-last-born, that we denote by $ l' $. Notice that $ l \notin c_{k+2}(l') $. Define $ \varphi(x) = l' $ and $\varphi(y) = l $. See Figure \ref{pic:Tk-extension-T-case2-thetwo}, right. It is easy to check that $\varphi$ has all the properties of Definition \ref{def:monotonic-embedding}, so $ T_{k+2} $ is an extension of~$ T $.   
\begin{figure}
		\centering
		\includegraphics[width=12cm]{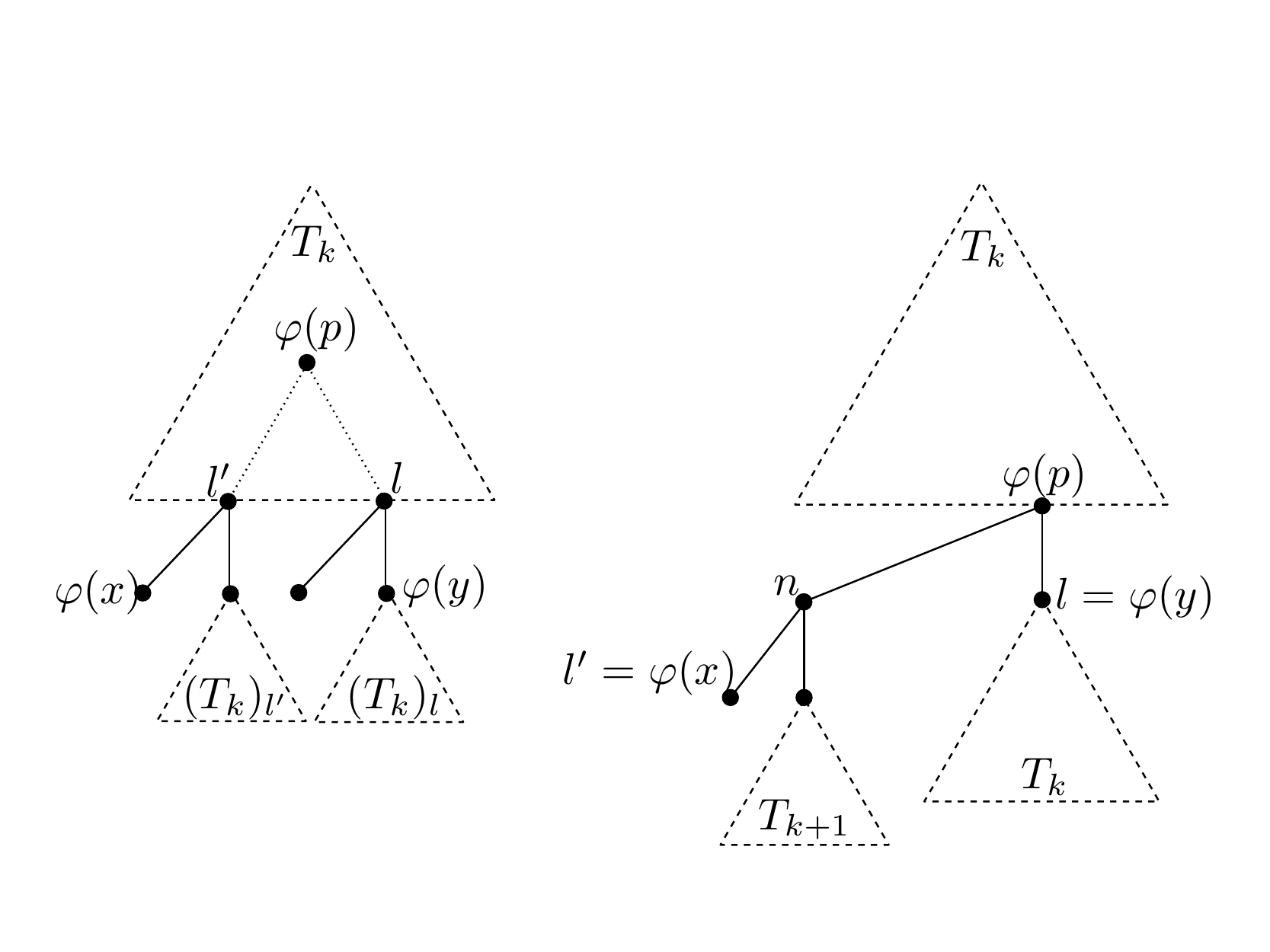}
		\vspace*{-1cm}
		\caption{Case 2 of the proof of Lemma \ref{lem:Tk-extension-T}} \label{pic:Tk-extension-T-case2-thetwo}
\end{figure}
	\end{proof}

	Now we can prove the main theorem of the section.
	
	\begin{theorem}
		\label{thm:B=D}
		The class of derived graphs is the
		same as the class of Burling graphs.
	\end{theorem}
	
	\begin{proof}
		Suppose that $ H $ is a Burling graph. So $ H $ is an induced
		subgraph of some $ G_k $ which is a fully derived graph by Lemma
		\ref{lem:Gk-is-derived-from-Tk}. Thus $ H $ is a derived graph.
		
		Now suppose that $ H $ is derived from a tree $ T $. By Lemma
		\ref{l:chooseT}, we may assume that every non-leaf vertex in $ T $
		has exactly two children. Notice that $ H $ is an induced subgraph
		of $ G $, the graph fully derived from $ T $. By Lemma
		\ref{lem:Tk-extension-T}, there exists $ k $ such that $ T_k $ is an
		extension of $ T $. Moreover, by Lemma
		\ref{lem:Gk-is-derived-from-Tk}, $ G_k $ is the graph fully derived
		from $ T_k $. Thus by Lemma \ref{lem:extension-induced-subgraph},
		$ G $ is an induced subgraph of $G_k $, and thus it is a Burling
		graph. Therefore, so is~$ H $.
	\end{proof}
	
	Theorem \ref{thm:B=D} enables us to interchangeably use the
        words Burling graphs or derived graphs for referring to this
        class.  The advantage of derived graphs to the classical
        definition of Burling graphs is that thanks to the tree
        structure, we can study the behavior of the stable sets much
        better. The Burling tree captures in 
        an easier way both the structure of the
        stable sets, and the adjacency of vertices in Burling graphs. 
        Moreover, as we will show in the second part of this work, 
        the orientation gives us more information about the
        properties of this class of graphs.
	
	\section{Abstract Burling graphs}
        \label{sec:abs}

        In this section, we prove that Burling graphs can be defined
        as \emph{abstract Burling graph}, that are graphs arising from
        two relations defined on a set and satisfying a small number
        of axioms.

	\begin{definition} \label{def:BurlingSet} A \emph{Burling set}
          is a triple $ (S, \desc, \adj) $ where $ S $ is a non-empty
          finite set, $\desc$ is a strict partial order on $S$,
          $ \adj $ is a binary relation on $S$ that does not have
          directed cycles, and such that the following axioms hold:
          \begin{enumerate}[label=(\roman*)]
          \item if $ x \desc y $ and $ x \desc z $, then either
            $ y \desc z $ or $ z \desc y $, \label{item:descdesc}
          \item if $ x \adj y $ and $ x \adj z $, then either
            $ y \desc z $ or $ z \desc y $, \label{item:adjadj}
          \item if $ x \adj y $ and $ x \desc z $, then $ y \desc z
            $, \label{item:adjdesc}
          \item if $ x \adj y $ and $ y \desc z $, then either
            $ x \adj z $ or $ x \desc z $. \label{item:transitiveboth}
          \end{enumerate}
	\end{definition}
	
	Let us give an example of a Burling set. Let
        $ (T, r, \lastBorn, \choosePath) $ be a Burling tree, and set
        $ V = V(T) $. For $ x, y \in V $, we
        define $ x \desc y $ if and only if $ x $ is a proper
        descendant of $ y $ in $ T $ and $ x \adj y $ if and only if
        $ y \in \choosePath(x) $.  Note that $ x \adj y $ if and only
        if there is an arc from $ x $ to $ y $ in the oriented graph
        fully derived from $ (T, r, \lastBorn, \choosePath) $.
	
	We show that $ (V, \desc, \adj) $ forms a Burling set. First notice that the proper descendant relation on a rooted tree forms a strict partial order. Second, remember that by Lemma \ref{lem:DG-no-cycle-no-triangle}, the relation $ \adj $ has no directed cycles. Now we check the four axioms of Definition \ref{def:BurlingSet}. Let $ x $, $ y $, and $ z  $ be three elements of $ V $:
	
	Axiom \ref{item:descdesc}: Suppose that $ x \desc y $ and $ x \desc z $. So both $ y $ and $z $ are ancestors of $ x $ in $T $, so they are on the same branch and hence comparable with respect to~$ \desc $.
	
	Axiom \ref{item:adjadj}: Suppose that $ x \adj y $ and $ x \adj z $. So $ y, z \in \choosePath(x) $. Thus by definition, they are on the same branch and are comparable with respect to~$ \desc $.
	
	Axiom \ref{item:adjdesc}: Suppose that $ x \adj y $ and $ x \desc z $. So $ y \in \choosePath (x) $ and thus $ y $ is a descendant of $ p(x) $. On the other hand, $ z $ is an ancestor of $ x $, so it is an ancestor of $ y $ too, and it is different from $ y $. Hence $ y \desc z $. 
	
	Axiom \ref{item:transitiveboth}: Suppose that $ x \adj y $ and $ y \desc z $. Let $ l $ be the last-born of $ p(x) $. So $ y $ is a descendant of $ l $, and $ z $ is an ancestor of $ y $. Either $ z $ is a descendant of $ l $ too, in which case $ x \adj z $ or $ z $ is a proper ancestor of $ l $, in which case it is a proper ancestor of $ x $ too, i.e.\ $ x \desc z $.

	
	\begin{lemma} \label{lem:only-on-of-the-four}
		Let $ S $ be a Burling set, and let $ x, y \in S $. At most one of the following holds: 
		$ x \adj y $, $ y \adj x $, $ x \desc y $, or $ y \desc x $. In particular, $ \adj \cap \desc = \varnothing $. 
	\end{lemma}
	
	\begin{proof}
		Notice that if any of the four relations hold, then $ x \neq y $, because $ \desc $ is a strict partial order and $ \adj $ has no directed cycle of length 1. 
		
		First suppose that $ x \adj y $. Because $\adj $ has no directed cycles, we cannot have $ y \adj x $. Moreover, if $ x \desc y $, then by Axiom \ref{item:adjdesc} of Definition \ref{def:BurlingSet} we must have $ y \desc y $ which is a contradiction. If $ y \desc x $, then by Axiom \ref{item:transitiveboth}, we have either $ x \adj x $ or $ x \desc x $, in both cases, it is a contradiction.
		
		It just remains to check that $ x \desc y $ and $y \desc x  $ cannot happen simultaneously, which is clear by the definition of strict partial orders.
	\end{proof}
	
	\begin{lemma} \label{lem:union-of-the-two-relations}
		Let $ \Rl = \adj \cup \desc $. The relation $ \Rl $ has no directed cycle. In particular, $ \Rl $ has some minimal element which is therefore minimal for both $ \adj $ and $ \desc $. 
	\end{lemma}
	
	\begin{proof}
		Suppose for the sake of contradiction that there is a cycle in $\Rl $, and let $ x_1, x_2, \dots, x_n $ be a minimal cycle. 
		
		By definition, $ n \neq 1 $, and by Lemma \ref{lem:only-on-of-the-four}, $ n \neq 2 $. 
		
		Now suppose that $ n \geq 4 $. Notice that none of $ \adj $ and $ \desc $ has a directed cycle, thus there exists $ 1 \leq i \leq n $, such that $ x_i \adj x_{i+1} $ and $ x_{i+1} \desc x_{i+2} $ (summations modulo $n $). Hence by Axiom \ref{item:transitiveboth}, we must have either $ x_{i} \adj x_{i+2} $ or $ x_{i} \desc x_{i+2} $. In any case, $ x_i \Rl x_{i+2} $, which is in contradiction to the minimality of the chosen directed cycle. 
		
		Finally, suppose that $ n = 3 $. Up to symmetry, we have $ x_1 \adj x_2 $ and $ x_2 \desc x_3 $, and therefore by Axiom \ref{item:transitiveboth}, $ x_1 \Rl x_3 $. But because this is a cycle, we must have $ x_3 \Rl x_1 $. This is in contradiction to Lemma \ref{lem:only-on-of-the-four}. 
		
		So $ \Rl $ has no directed cycle. So there exists a
                minimal element in $ \Rl $ which is, by definition, a
                minimal element for both $ \adj $ and $ \desc $.
	\end{proof}
	
	We recall that in a given Burling set $ S$, and for an element $ s $ in $ S $, $ [s{\adj}] = \{t \in S : s \adj t \} $, and $ [s {\desc}] = \{ t \in S : s \desc t  \} $.  
	
	\begin{lemma} \label{lem:Bset-everything-on-a-branch}
		Let $ s $ be an element of a Burling set $ S $. Then there exists an ordering of the elements of $ [s{\adj}] $ such as $ u_1, u_2, \dotsm u_k $ and an ordering of the elements of $ [s{\desc}] $ such as $ v_1, v_2, \dots v_l $ such that $ u_1 \desc u_2 \desc \dots \desc u_k \desc v_1 \desc v_2 \desc \dots \desc v_l $.
	\end{lemma}  
	
	\begin{proof}
		By Axiom \ref{item:adjadj} of Burling sets, all the elements of $ [s{\adj}] $ form a chain $ u_1 \desc u_2 \desc \dots \desc u_k $. Moreover, by Axiom \ref{item:descdesc}, all the elements of $ [s{\desc}] $  also form a chain $ v_1 \desc v_2 \desc \dots \desc v_l $. Finally, $ u_k \desc v_1 $ follows from Axiom \ref{item:adjdesc} since $ s \adj u_k $ and $ s \desc v_1 $.  
	\end{proof}
	

	Let $ (S, \desc, \adj) $ be a Burling set. We define the
        oriented graph $ G $ \textit{derived} from
        $ (S, \desc, \adj) $ as the oriented graph on vertex-set $ S $ such that
        for $ x, y \in S $, $ xy \in A(G) $ if and only if
        $ x \adj y $. We denote $ G $ by $ \Gr(S) $, and we say that
        $ G $ is an \emph{abstract Burling graph}.
	
	Notice that if $ S $ is a Burling set and $ G = \Gr (S) $, then for every induced subgraph $  G' $ of $ G $, $ S' = V(G') $ as a subset of $ S $ is itself a Burling set with inherited relations $ \desc $ and $ \adj $, and moreover $ G' = \Gr(S') $.

	\subsection*{Equality of abstract Burling graphs and Burling graphs}
	\begin{lemma} \label{lem:derived-is-abstractB}
		Every oriented derived graph is an abstract Burling graph.
	\end{lemma}
	\begin{proof}
		We checked after Definition \ref{def:BurlingSet} that if $ T $ is a Burling tree, then $ V = V(T) $ forms a Burling set, and from there, it follows easily that the graph fully derived from $ T $ is exactly $ \Gr (V) $. Thus every fully derived Burling graph is an abstract Burling graph. Moreover, since abstract Burling graphs form a hereditary class, every derived graph is an abstract Burling graph.
	\end{proof}
	
	\begin{lemma} \label{lem:abstractB-is-derived}
		Let $ G $ be an oriented graph. If $ G = \Gr (S) $ for some Burling set $ S $, then $ G $ is an oriented derived graph.
	\end{lemma}
	
	\begin{proof}
		We prove the following statement by induction on the number of elements of $ S  $.
		
		\begin{statement}
			There exists a Burling tree $ (T, r, \lastBorn, \choosePath) $ such that $ S \subseteq V(T) $, and for every two distinct elements $ x $ and $ y $ in $ S $:
			\begin{enumerate}[label=(\roman*)]
				\item $ x \desc y $ if and only if $ x $ is a descendant of $ y $ in $ T$,
				\item $ x \adj y $ if and only if $ y \in \choosePath(x)  $ in $ T $.
			\end{enumerate}
		\end{statement}
		
		If $ |S| = 1 $, then the result obviously holds. Suppose that the statement holds for every Burling set on at most $ k-1 $ elements, and let $ S $ be a Burling set on $k \geq 2 $ elements. 
		
		Let $ v \in S $ be a minimal element of $ \adj \cup \desc $ which exists by Lemma \ref{lem:union-of-the-two-relations}. Set $ S' = S \sm \{v\} $. By the induction hypothesis, there exists a Burling tree $ (T', r', \lastBorn ', \choosePath ') $ such that $ S' \subseteq V(T') $ and the two properties of the statement hold.
		
		Now let $ [v{\adj}] = \{u_1, u_2, \dots, u_m \} $ and
                $ [v{\desc}] = \{w_1, w_2, \dots, w_n \} $ (both
                possibly empty). By Lemma
                \ref{lem:Bset-everything-on-a-branch}, suppose without
                loss of generality that $ u_1 \desc u_2 \desc \dots
                \desc u_m \desc w_1 \desc w_2 \desc \dots \desc w_n
                $. Thus by the induction hypothesis, they appear on a
                same branch of $ T' $. So from the root to the leaf,
                they appear in this order: $ w_n, \dotsm w_1, u_n,
                \dots, u_1 $. Now we consider two cases: 
		
		\textit{Case 1:} $ [v{\desc}]  = \varnothing $. In this case, add a parent $ r $ to $ r' $ and define $ \lastBorn (r) = r' $. Then add $ v $ as a child of $ r$. If $ [v {\adj}] = \varnothing $, then define $\choosePath (v) = \varnothing $. Otherwise, let $ P $ be the set of vertices on the path between $ r' $ and $ u_1 $, including both of them, and define $ \choosePath (v) = P $. Call this new Burling tree $ T $. 
		
		\textit{Case 2:} $ [v {\desc}]  \neq \varnothing $. In this case, if $ w_1 $ is a leaf, and hence $ [v{\adj}] = \varnothing $, then add $ v $ as a last-born child of $ w_1 $ and define $\choosePath (v) = \varnothing $. If $ w_1 $ is not a leaf, then add $ v $ as a non-last-born child of $ w_1 $. If $ [v{\adj}] = \varnothing $, define $\choosePath (v) = \varnothing $. Otherwise, let $ P $ be the set of vertices on the path between $\lastBorn (w_1) $ and $ u_1 $, and define $ \choosePath (v) = P $. Call the obtained Burling tree~$ T $.
		
		In both cases, we obviously have $ S \subseteq V(T) $, so it remains to prove the two properties of the statement. For any two distinct elements of $ S $ which are both different from $ v$, the result follows from the induction hypothesis. So consider $ v $ and an element $ u $ of $ S $ different from $ v $. Notice that by minimality of $ v $ with respect to both relations, we have neither $ u \adj v $ nor $ u \desc v $ in $ S$, and by the construction of $ T $, in both cases, $ v $ is not in $ \choosePath(u) $, and it has no descendant, so in particular, $ u $ is not a descendant of $ v $. Moreover, by construction of $ T $ in both cases, if $ v \desc u $ in $ S $, then $ v $ is a descendant of $ u $ in $ T $, and if $ v \adj u $ in $ S$, then $ u \in \choosePath (v) $ in $ T $. 
		
		Now suppose that $ x $ is an element of $ S $, and in $ T $, $ x $ is an ancestor of $ v $, and thus we are necessarily in case 2. We prove that $ v \desc x $. If $ x  =w_1$, then the result is immediate. Otherwise, $ x $ is an ancestor of $ w_1 $. Thus by the induction hypothesis, $ w_1  \desc x $. On the other hand, $ v \desc w $. Since $\desc $ is a strict partial order, $ v \desc x $.
		
		Finally, suppose that $ x $ is an element of $ S $ and in $ T $, $x \in \choosePath (v) $. We show that $ v \adj x $ in $ S $. From $x \in \choosePath (v) $, we know that $ x $ is a vertex among the vertices of the path from the last-born of $ w_1 $ to $ u_1 $. If $ x = u_1 $, then the result is immediate. If not, we have $ v \adj u_1 $ and $ u_1 \desc x  $. So by Axiom (i) of Definition \ref{def:BurlingSet}, either $ v \adj x $ or $ v \desc x $. But the latter is not possible because otherwise from $ v \desc x $ and the fact that $ v \neq x $, we know that $ x $ is either $ w_1 $ or it is an ancestor of $ w_1 $ in $T$. But this is not possible, because $ x \in \choosePath (v) $.
		
		To complete the proof we notice that $ G $ is exactly the subgraph of the graph derived from $ T $, induced by the vertices of $ S $.
	\end{proof}
	
	\begin{theorem} \label{lem:abstractB=derived=Burling}
		The class of abstract Burling graphs is equal to the class of derived graphs, and therefore to the class of Burling graphs.
	\end{theorem}
	
	\begin{proof}
		The proof follows directly from Lemmas \ref{lem:derived-is-abstractB} and \ref{lem:abstractB-is-derived}.
	\end{proof}
	
		We remark here that even though the classes of graphs derived from Burling sets and the graphs derived from Burling trees are the same, there is no immediate one-to-one correspondence between Burling sets and Burling graphs. Burling sets do not need notions equivalent to root and last-born. This is what makes them a general object to work with. On the other hand, derived graphs, having all these specific notions, provide strong tools to deduce structural results, as we will see in the second part of this work.
	
                \subsection*{Topological orderings}
                
        We here make several remarks about topological orderings of
        the vertices in Burling graphs, and BFS and DFS algorithms on them.  We do not really
        need these easy observations (and therefore omit their
        straightforward formal proofs), but we believe that they help
        to understand the next section.

        It was observed in Lemma~\ref{lem:DG-no-cycle-no-triangle}
        that oriented derived graphs have no directed cycles.  And in
        Lemma~\ref{lem:union-of-the-two-relations}, we go further and
        observe that the union of the relations $\desc $ and $\adj$
        has no directed cycles (which is easy to see directly on an
        oriented graph derived from a tree). This means that when an
        oriented graph $G$ is derived from a tree $T$, if one orients
        every edge of $T$ from the root to the bottom, then the
        oriented graph $G^*$ on $V(T)=V(G)$ with the union of arcs
        from $T$ and from $G$ has no directed cycle.

        This implies that there should exist a topological ordering of
        $G^*$. And indeed, there is a natural way to find one: with
        BFS applied to $T$ (starting at the root, and with priority
        given to non-last-borns).  If we denote by $\score(v)$ the
        opposite of the number given by BFS to each vertex, we have the
        following: if $u \desc v$ or $u \adj v $, then
        $\score(u) < \score (v)$. See Figure~\ref{fig:topo}.

        Now, in $G^*$, change the orientation of every arc of $G$ (but
        keep the arcs of $T$ from root to leaves).  Again, it is easy
        to check that there is no directed cycle, so there should
        exist a topological ordering again.  This time, an ordering may
        be obtained with DFS (starting at the root and with priority
        given to the last-born).  If denote by $\antiscore(v)$ the
        opposite of the number given by DFS to each vertex, we have
        the following: if $v \desc u$ or $u \adj v $, then
        $\antiscore(u) < \antiscore (v)$.  An example is represented in
        Figure~\ref{fig:topo}.

        \begin{figure}
          \centering

          \includegraphics{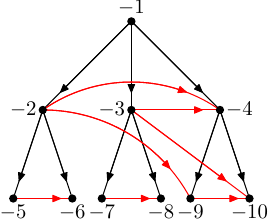} \hspace{1cm}
          \includegraphics{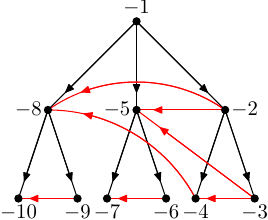}
          \caption{Functions $\score$ and $\antiscore$ viewed as
            topological orderings\label{fig:topo}. Every arc is from
          a number to a smaller number.} 
        \end{figure}

        We sum up the main properties of $\score$ and $\antiscore$, as
        defined in this paragraph, in the next lemma.  As we will see
        in the next section, in geometrical interpretations of Burling
        graphs, there are natural geometrical counterparts of the
        functions $\score$ and $\antiscore$, with each time a similar
        lemma.

        \begin{lemma}
          \label{l:scoreDerived}
          If $A \desc B$, then $\score(A) < \score(B)$ and
          $\antiscore(A) < \antiscore(B)$.  If $A \adj B$, then
          $\score(A) < \score(B)$ and $\antiscore(A) > \antiscore(B)$.
        \end{lemma}

        \section{Burling graphs as intersection graphs}
        \label{sec:geom}

        In this section, we define three classes of graphs: \emph{strict frame
          graphs}, \emph{strict line segment graphs}, and \emph{strict box graphs}. We show that
        they are all equal to the class of Burling graphs.

        \subsection*{Strict frame graphs}

        A \textit{frame} in $\mathbb{R}^2$ is the boundary of an
        axis-aligned rectangle. Intersection graphs of frames are
        called \textit{frame graphs}. Frame graphs clearly form a
        hereditary class of graphs. The class of \textit{restricted
          frame graphs}, defined in~\cite{Chalopin2014}
        (Definition~2.2.), is a subclass of frame graphs. They are the
        frame graphs with some extra restrictions.

        \begin{definition}
          \label{def:rfg}
          A set of frames in the plane is \emph{restricted} if it has
          the following restrictions:
          \begin{enumerate}[label=(\roman*)]
          \item there are no three frames, which are mutually
            intersecting (in other words, the intersection graph of
            the frames is triangle-free),
          \item corners of a frame do not coincide with any point of
            another frame,
          \item the left side of any frame does not intersect any
            other frame,
          \item if the right side of a frame intersects a second
            frame, this right side intersects both the top and bottom
            of this second frame,
          \item if two frames have non-empty intersection, then no
            frame is (entirely) contained in the intersection of the
            regions bounded by the two frames.  If frames $ A $ and
            $ B $ intersect as in Figure~\ref{pic:only-possible-rfg},
            we say that frame~$ A $ \textit{enters} frame~$ B $.
          \end{enumerate}
        \end{definition}

        The only possibility for two frames to intersect with
        these restrictions is shown in Figure~\ref{pic:only-possible-rfg}. 
        In such case, we say that the frame~$ A $ \emph{enters} the frame~$ B $.

        \begin{figure}[h!]
          \begin{center}
            \centering
            \vspace*{-1.2cm}
            \includegraphics[width=7cm]{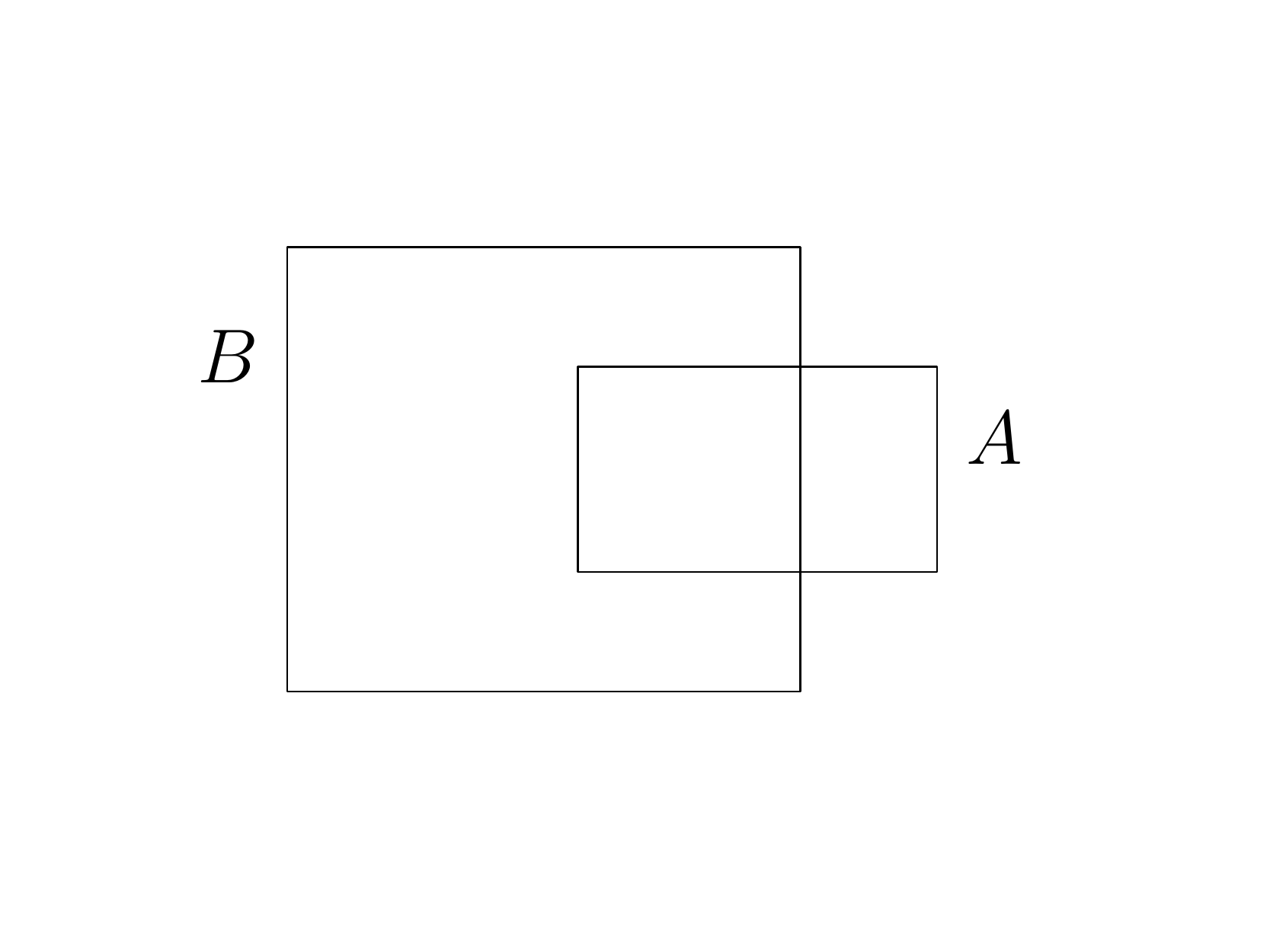}
            \vspace*{-1.5cm}
            \caption{Intersection of two frames in restricted frame
              graphs} \label{pic:only-possible-rfg}
          \end{center}
        \end{figure}

		\begin{figure}
			\begin{center}
				\vspace*{-2.5cm}
				\includegraphics[width=12cm]{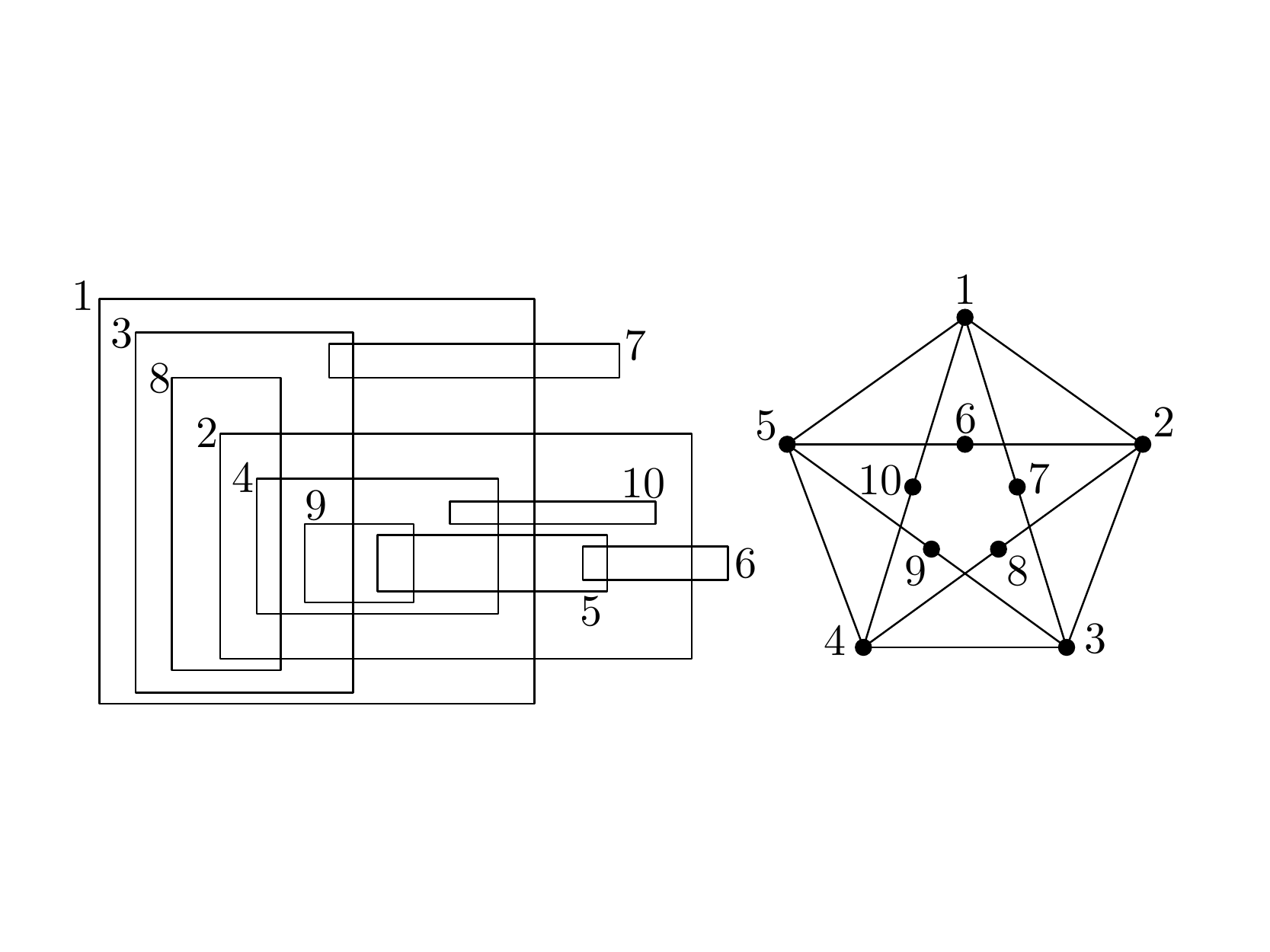}
				\vspace*{-2.4cm}
				\caption{\footnotesize A graph obtained from $ K_5 $ by subdividing some edges and its
					presentation as a restricted frame
					graph} \label{pic:k5sd-rfg}
			\end{center}
			\begin{center}
				\vspace*{-1.6cm}
				\includegraphics[width=12cm]{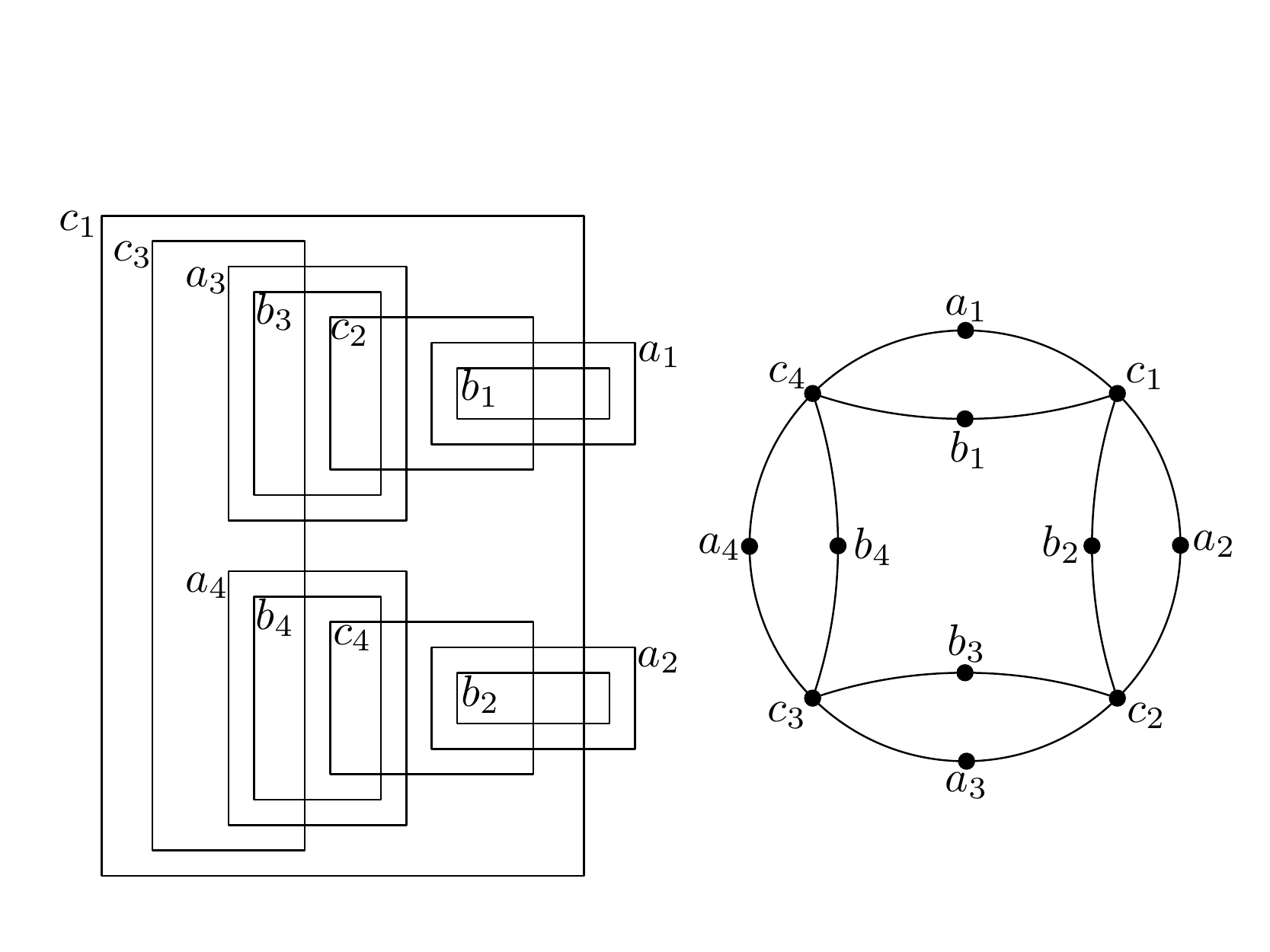}
				\vspace*{-.7cm}
				\caption{\footnotesize A necklace and its presentation
					as a restricted frame
					graph} \label{pic:necklace-rfg}
			\end{center}
			\begin{center}
				\vspace*{-2.1cm}
				\includegraphics[width=12cm]{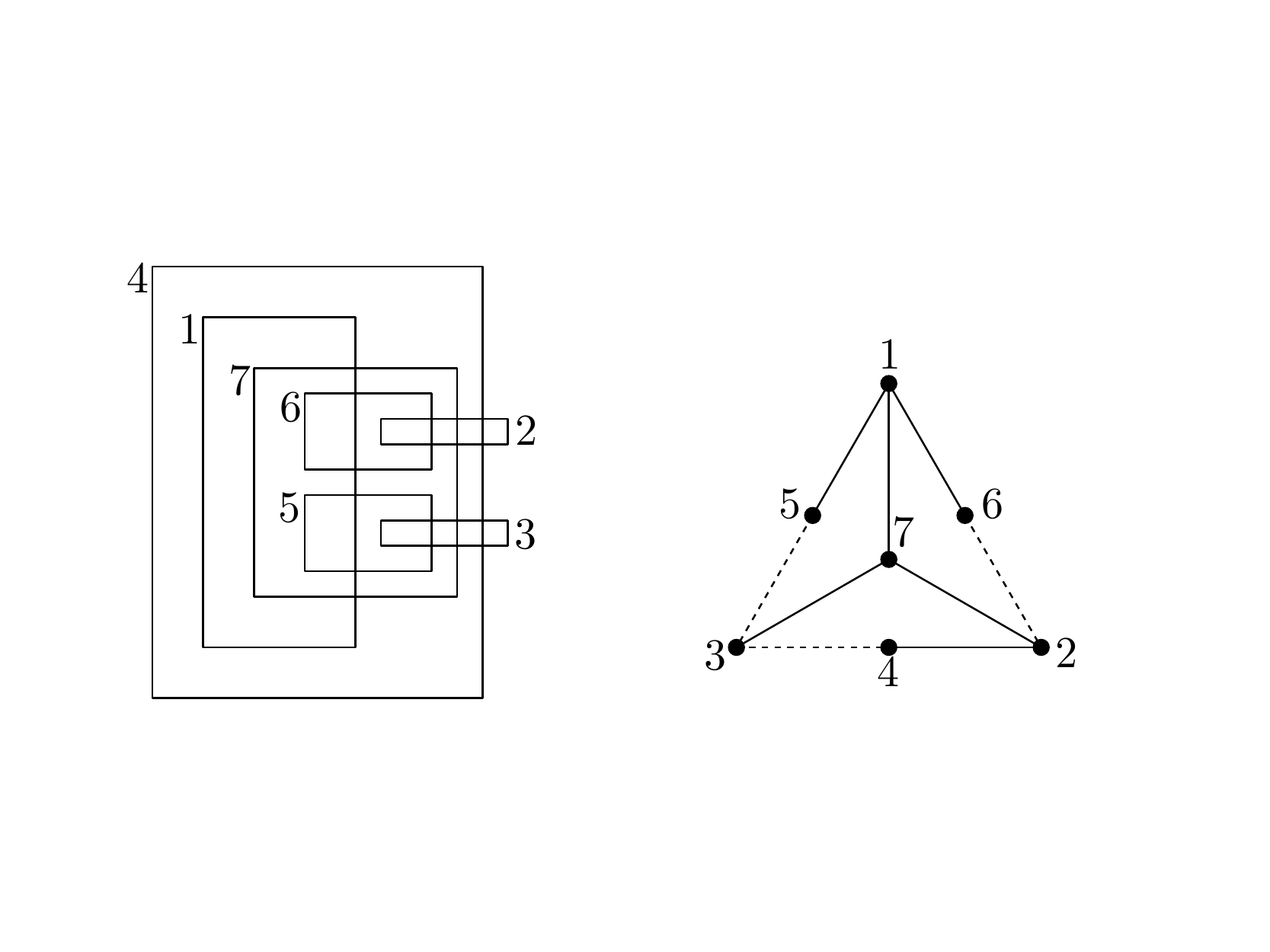}
				\vspace*{-2.5cm}
				\caption{\footnotesize A wheel and its presentation as
					a restricted frame graph. It is easy to see that
					subdividing the dashed edges yields to restricted frame
					graphs (see also Appendix A of~\cite{Chalopin2014}).} \label{pic:wheel-rfg}
			\end{center}
		\end{figure}

		 A \emph{restricted frame graph} is the intersection graph of a
		restricted set of frames in the plane.  An \emph{oriented
			restricted frame graph} is a frame graph such that every
		edge $ AB $ is oriented from $ A $ to $ B $ when frame~$ A $
		enters frame~$ B $.
		
        See Figures \ref{pic:k5sd-rfg}, \ref{pic:necklace-rfg} and
        \ref{pic:wheel-rfg} for some examples of restricted frame
        graphs. It is worth noting that in the second part of this work,
        we prove that non of these graphs are Burling graphs.

		Now we introduce the class of \emph{strict frame graphs}, a subclass
of restricted frame graphs, and show that it is equal to the
class of Burling graphs.

            \begin{definition}
              \label{def:strictfg}
              A restricted set of frames in the plane is \emph{strict}
              if for any two frames $ A $ and $ B $ such that $ A $ is
              entirely inside $ B $, when a frame $ C $ intersects
              both, $ C $ enters both $ A $ and $ B $. See Figure
              \ref{pic:super-rfg}.
            \end{definition}
            \begin{figure}[t]
              \begin{center}
                \vspace*{-2.5cm}
                \includegraphics[width=8cm]{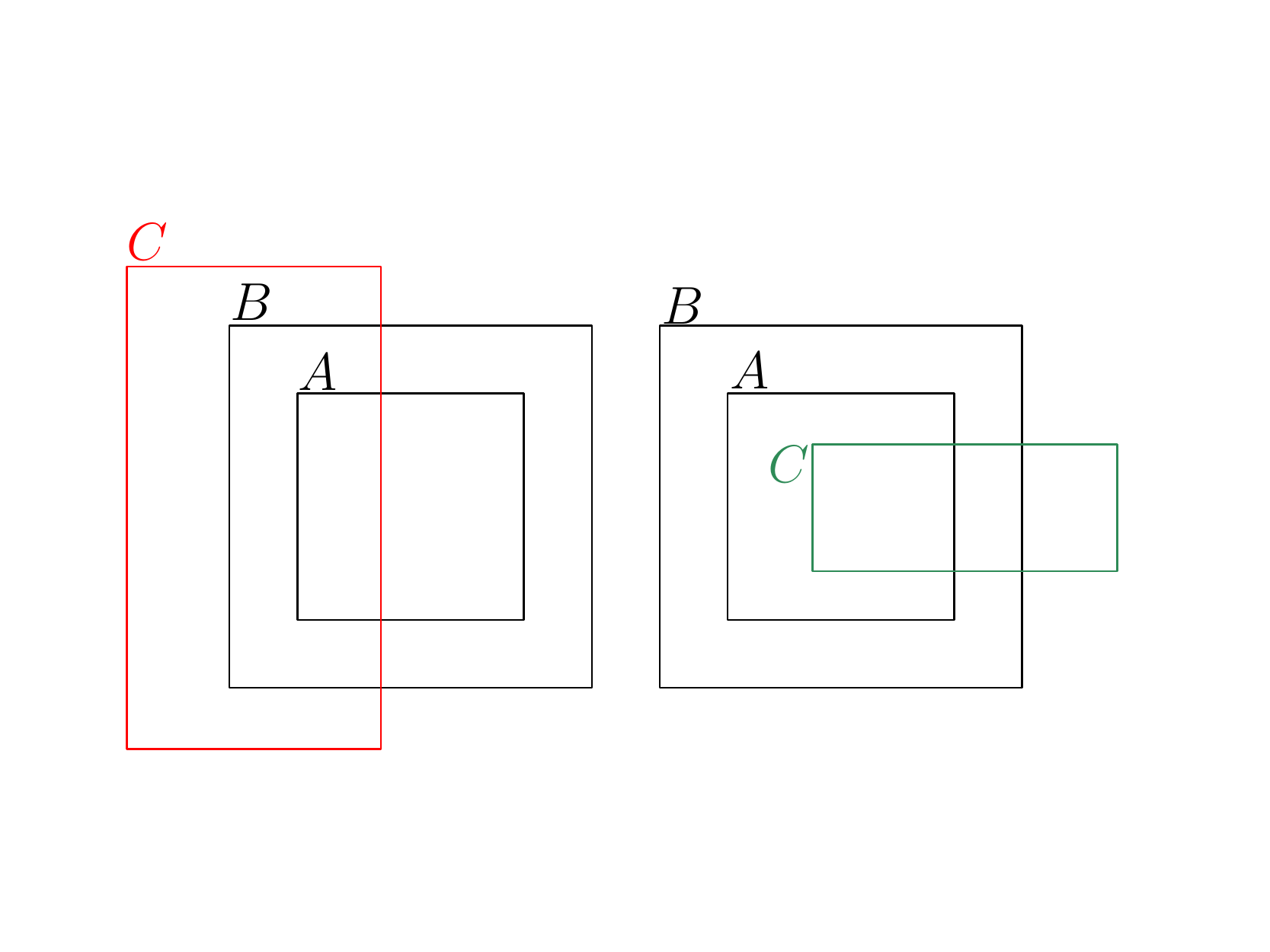}
                \vspace*{-1.4cm}
		\caption{\footnotesize Left: the forbidden structure
                  in strict frame graph, right: the allowed
                  structure} \label{pic:super-rfg}
              \end{center}
            \end{figure}

            A \emph{strict frame graph} is the intersection graph of a
            strict set of frames in the plane. An \emph{oriented
              strict frame graph} is a strict frame graph, oriented
            as an oriented restricted frame graph.

            Let $ F $ be a non-empty finite strict set of frames in
            the plane. Define $ A \desc B $ if and only if $ A $ is
            entirely inside $ B $. Define $ A \adj C$ if and only if
            $ A $ enters $ C $.  We denote by $ A \interior $ the area
            that frame $ A $ encloses.  Two frames
            are \emph{comparable} if
            $ A \interior \cap B \interior \neq \varnothing $ and
            \emph{incomparable} otherwise.  Note that in a strict set of frames, 
            two frames are
            comparable if and only if one of them enters the other or
            is inside the other.

            We denote by $\score(A)$ the vertical length of a frame
            $A$ and by $\antiscore(A)$ the maximum real number $x$
            such that $x$ is the $x$-coordinate of a point of $A$.

            \begin{lemma}
              \label{l:scoreFrame}
              If $A \desc B$, then $\score(A) < \score(B)$ and
              $\antiscore(A) < \antiscore(B)$.  If $A \adj B$, then
              $\score(A) < \score(B)$ and
              $\antiscore(A) > \antiscore(B)$.
            \end{lemma}

            \begin{proof}
              Obvious from the definitions.
            \end{proof}

          \begin{lemma}
            \label{lem:Frames-are-Abstract-B}
            For every non-empty finite and strict set of frames $ F $,
            the triple $ (F, \desc, \adj) $ forms a Burling set.
          \end{lemma}

          \begin{proof}
            First, $ \desc $ is obviously transitive and asymmetric,
            so it is a strict partial order. Moreover, by
            Lemma~\ref{l:scoreFrame}, the relation $ \adj $ cannot
            have any directed cycle. Now we prove that the four axioms
            hold. Now we prove that the four axioms of Burling sets
            hold.

            Axiom \ref{item:descdesc}: If $ A \desc B $ and
            $ A \desc C $, $ C\interior $ and $ B \interior $ both
            contain $ A $ and thus their intersection is not empty, so
            $ B $ and $ C $ are comparable. Now we cannot have
            $ B \adj C $ or $ C \adj B $, because it contradicts item
            (v) of Definition \ref{def:rfg}. Thus either $ B \desc C $
            or $ C \desc B $.

            Axiom \ref{item:adjadj}: If $ A \adj B $ and $ A \adj C $,
            then again
            $ B \interior \cap C \interior \neq \varnothing $, so
            $ B $ and $ C $ are comparable. But because $ F $ is
            triangle-free, we cannot have $B \adj C $ or $ C \adj B
            $. So either $ B \desc C $ or $ C \desc B $.
	
            Axiom \ref{item:adjdesc}: if $ A \adj B $ and
            $ A \desc C $, then by Lemma~\ref{l:scoreFrame},
            $\antiscore(B) < \antiscore(A) < \antiscore(C)$.  Since
            there are points of $ A \interior $ which are in both
            $ B \interior $ and $ C \interior $, $ B $ and $C $ are
            comparable. Since $\antiscore(B) < \antiscore(C)$, by
            Lemma~\ref{l:scoreFrame}, we cannot have $ B \adj C $ or
            $ C \desc B $. Moreover, $ C \adj B $ contradicts
            the restriction of Definition \ref{def:strictfg}. Thus
            $ B \desc C $.

            Axiom \ref{item:transitiveboth}: If $ A \adj B $ and
            $ B \desc C $, then by Lemma~\ref{l:scoreFrame},
            $\score(A) < \score(B) < \score(C)$.  There are points of
            $ A \interior $ which are inside $ C \interior $. So $ A $
            and $ C $ are comparable. Since $\score(A) < \score(C)$,
            by Lemma~\ref{l:scoreFrame}, we cannot have $ C \adj A $
            or $ C \desc A $. So, $ A \adj C $ or $ A \desc C $.
            \end{proof}

            \begin{lemma}
              \label{lem:Burling-is-sfg} Every Burling graph is a
              strict frame graph.
            \end{lemma}

            \begin{proof}
              If $ G $ is a derived graph, then by theorem
              \ref{thm:B=D} it is an induced subgraph of a graph
              $ G_k $ in the Burling sequence. It is easy to check
              that in the geometrical representation of the Burling
              sequence in~\cite{DBLP:journals/dcg/PawlikKKLMTW13} one
              never creates the forbidden structure of Definition
              \ref{def:strictfg}.  One can see the construction of the
              graphs in the Burling sequence as restricted frame
              graphs
              in~\cite{DBLP:journals/dcg/PawlikKKLMTW13,Chalopin2014},
              and check that in their construction, the forbidden
              constraint of Definition~\ref{def:strictfg} does not
              happen. Moreover, we notice restriction of frames to an
              induced subgraph, does not create any of the forbidden
              constraints.
            \end{proof}

            \begin{theorem} \label{thm:strict-frames=BG}
              The class of strict frame graphs is equal to the class
              of Burling graphs.
            \end{theorem}

            \begin{proof}
              By Lemma \ref{lem:Frames-are-Abstract-B}, every frame
              graph, is the underlying graph of an abstract Burling
              graph, and therefore by Theorems
              \ref{lem:abstractB=derived=Burling}, a Burling
              graphs. This, along with Lemma \ref{lem:Burling-is-sfg},
              completes the proof.
            \end{proof}

          \subsection*{Strict line segment graphs}

          Let $ l $ be a non-vertical line segment in $\mathbb{R}^2$. We
          can characterize $ l $ by $ y = ax +b $ for
          $ x \in [\alpha, \beta] $. The number $ a $ is the
          \textit{slope} of $ l $. We say that $ l $ has
          \emph{positive slope} if $ a $ is a finite positive number
          (in which case $l$ is neither horizontal nor vertical). We
          denote the interval $ [\alpha, \beta] $ by $ \Xspan (l) $,
          and the interval
          $ [ a\alpha +b, a \beta +b ] $ by
          $ \Yspan (l)$. Finally, for $ l $ with positive slope, we
          define the \textit{territory} of $ l $ to be the
          unbounded polyhedron defined by
          $ y \geq ax +b $ and $ y  \in \Yspan (l) $.
          We denote the territory of $ l $ by $ \mathcal T(l) $. See
          Figure~\ref{territory}.

          \begin{figure}[h!]
            \centering 
            \vspace*{-1.7cm}
            \includegraphics[width=7cm]{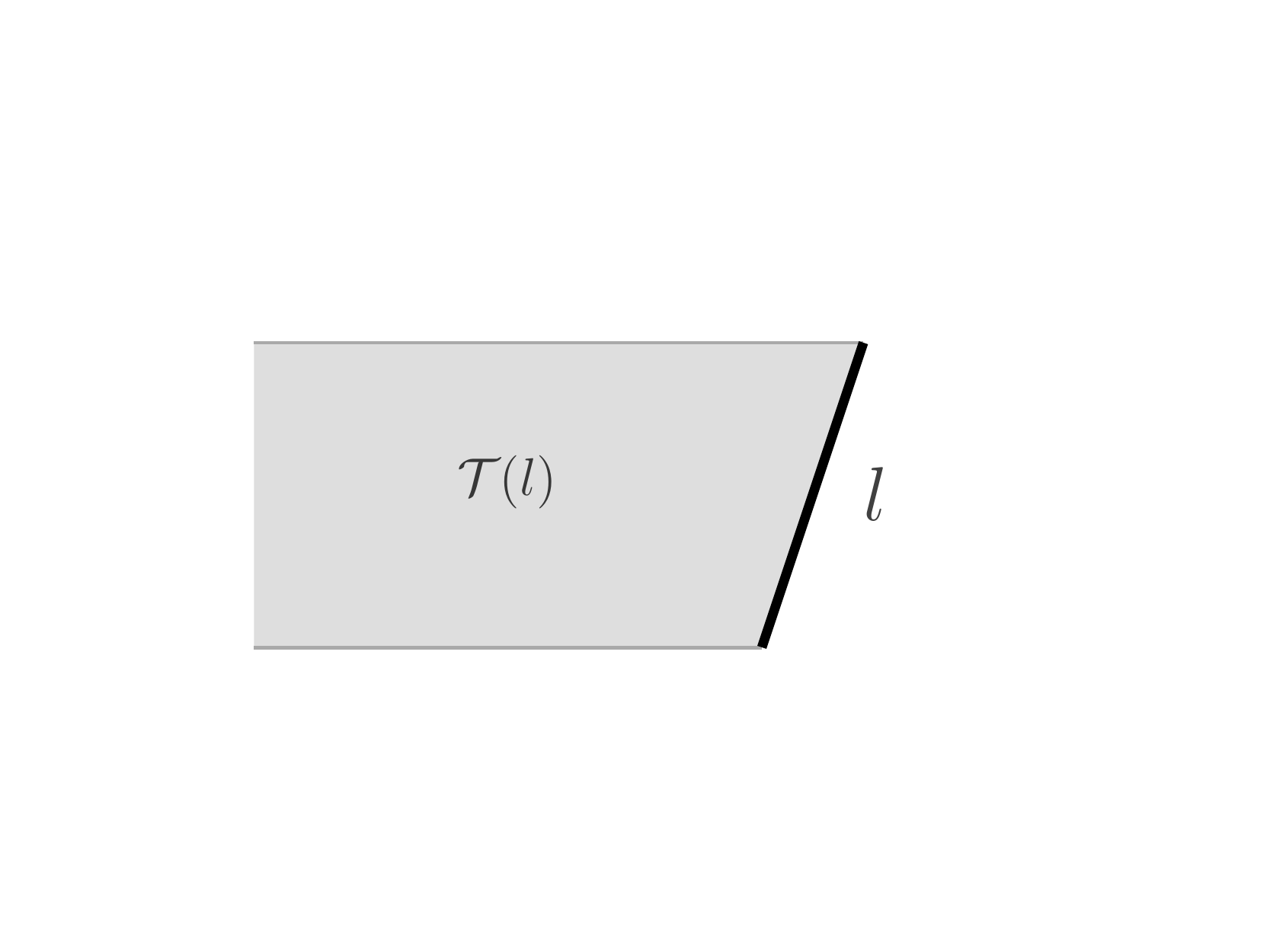}
            \vspace*{-1.7cm}
            \caption{The gray area is the territory of segment $ l
              $.} \label{territory}
          \end{figure}

          \begin{definition} \label{def:slsg} Let $ L $ be a finite
            set of line segments in $\mathbb{R}^2$. We call $ L $ a
            \emph{strict} set of line segments if the following hold:
            \begin{enumerate}[label=(\roman*)]
            \item all the segments in $ L $ have positive slopes,
            \item no end-point of any line segment lies on another
              line segment,
            \item there exist no three pairwise intersecting
              line segments in $L $ (in other words, the intersection
              graph of $L$ is triangle-free),
            \item for any two non-intersecting line segments $ l $ and
              $ m $, if there exists a point $ p $ of $ m $ such that
              $ p \in \mathcal T(l) $, then $ m $ is entirely in
              $ T(l) $ and $\Yspan(m) \subsetneq \Yspan(l)$,
            \item if $ l $ and $ k$ are two intersecting segments,
              then there are no segments entirely inside
              $ \mathcal T (l) \cap \mathcal T (k)$,
            \item If $ k $ and $l $ are two intersecting segments such
              that the slope of $ k $ is less than the slope of $ l $,
              then $ \Yspan (k) \subsetneq \Yspan (l) $ and
              $ \Xspan (l) \subsetneq \Xspan (k) $ such that the maximum 
              of $\Xspan (l) $ is strictly less than the maximum of $ \Xspan (k) $. 
              See Figure~\ref{two-inters-segments},
              \begin{figure}[h!]
                \centering
                \vspace*{-1cm}
                \includegraphics[width=7cm]{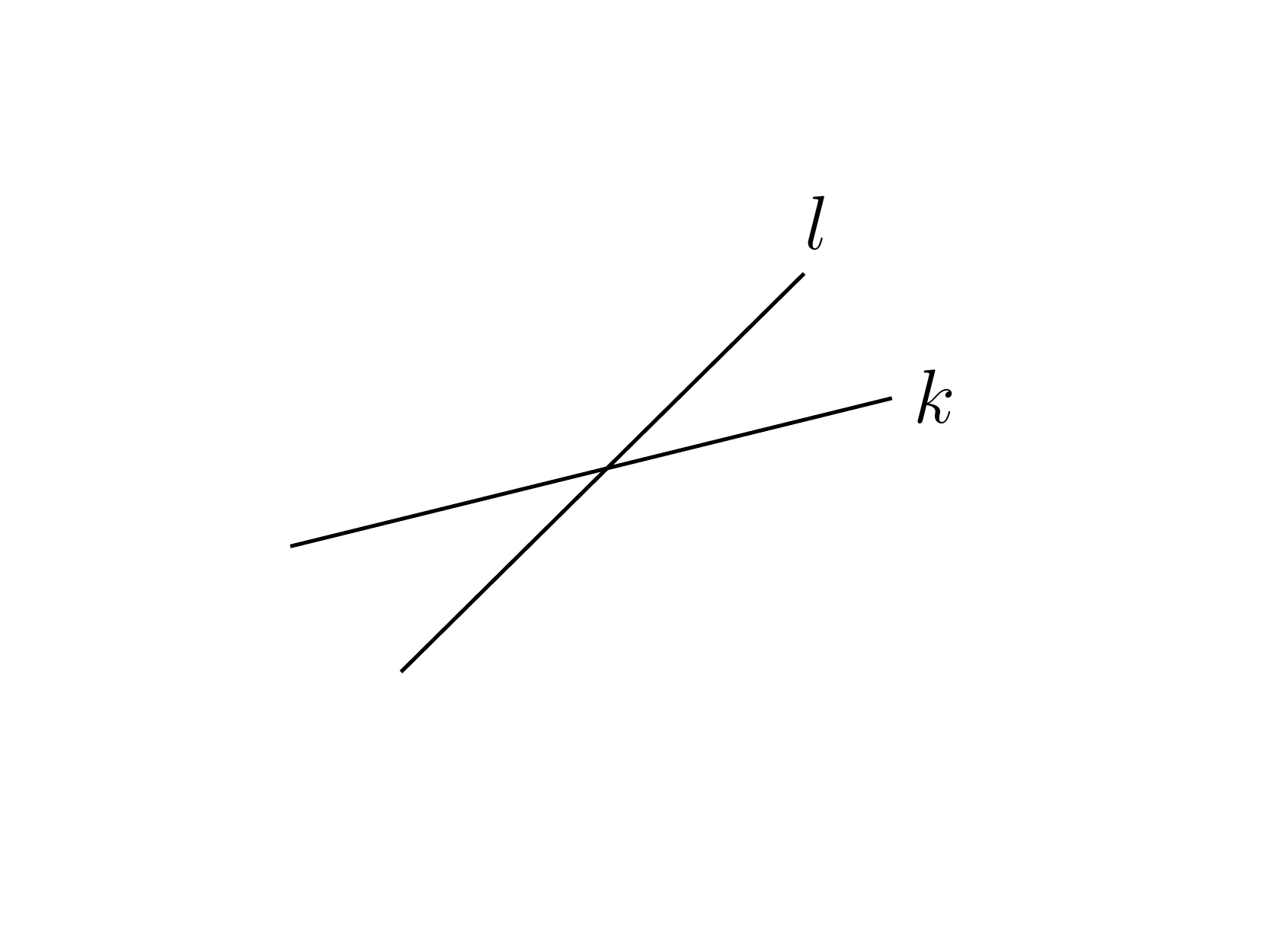}
                \vspace*{-1.5cm}
                \caption{two intersecting segments}
                \label{two-inters-segments}
              \end{figure}

            \item for any two non-intersecting line segments $ l $ and
              $ m $ such that one is in the territory of the other, if
              a line segment $ k $ intersects both of them, then the
              slope of $ k $ is strictly less than both the slope of
              $l $ and the slop of $ m$, as illustrated in Figure
              \ref{fig:forbidden-line-segment}.
              \begin{figure}[h!]
                \vspace*{-1cm} \centering
                \includegraphics[width=9cm]{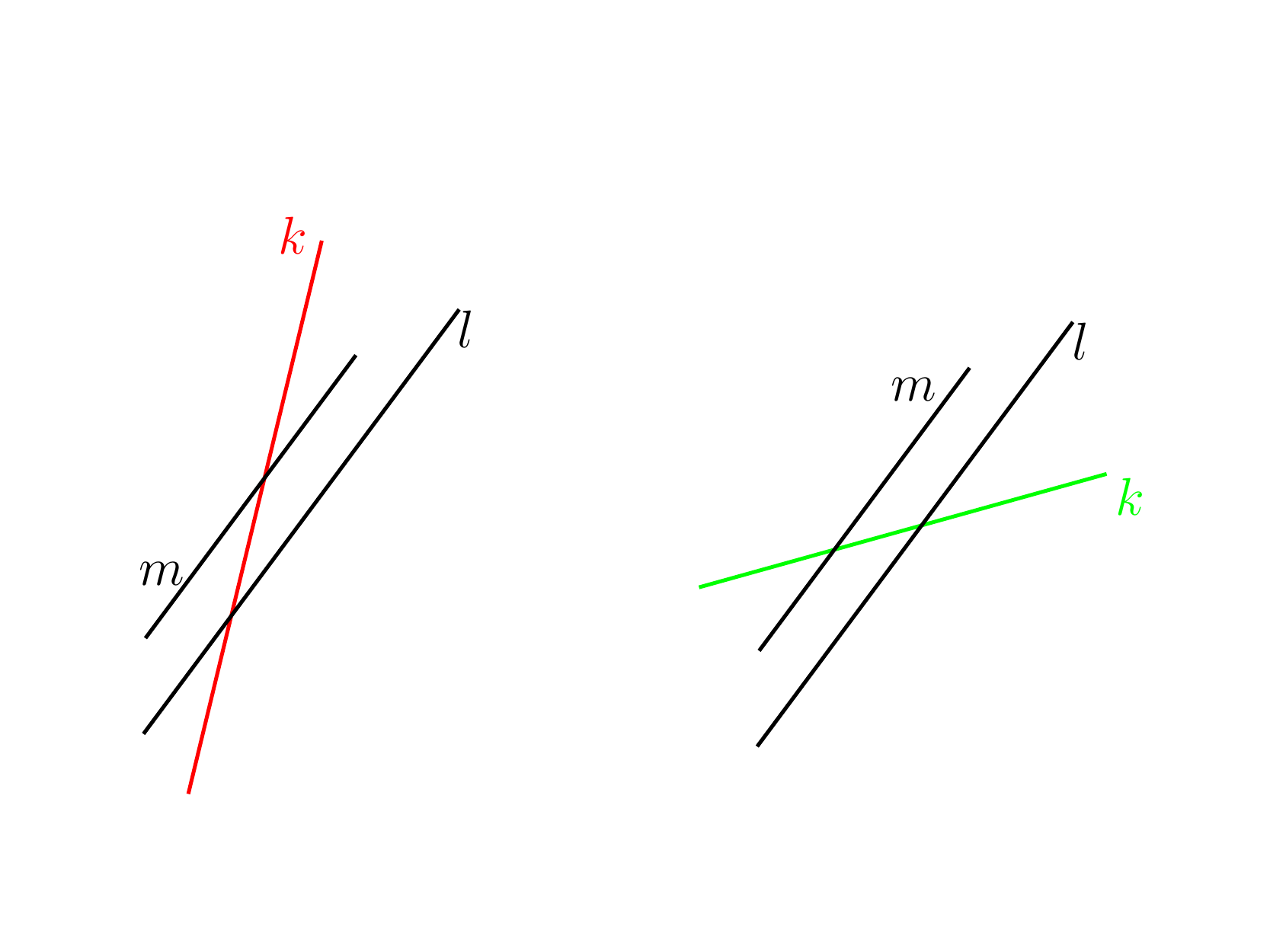}
                \vspace*{-1.2cm}
                \caption{Left: forbidden and right: allowed structure
                  in constraint
                  (vii).} \label{fig:forbidden-line-segment}
              \end{figure}
            \end{enumerate}
          \end{definition}

          Let $ L $ be a strict finite set of line segments
          in the plane. Define $ l \desc k $ if and only if $ l $ is
          in the territory of $ k $, and define $ l \adj k$ if and
          only if $ l $ and $ k $ have non-empty intersection and the
          slope of $ l $ is less than the slope of $ k $.

          Note that by Constraint (ii) of Definition \ref{def:slsg},
          intersecting line segments must have distinct slopes.

          \begin{lemma}
            \label{l:remark}
            If for two non-intersecting line segments $ l $ and $ m $
            we have
            $ \mathcal T (l) \cap \mathcal T (m) \neq \varnothing $,
            then $ \mathcal T (l) \subseteq \mathcal T (m)$ or
            $ \mathcal T (m) \subseteq \mathcal T (l)$.
          \end{lemma}

          \begin{proof}
            Since
            $ \mathcal T (l) \cap \mathcal T (m) \neq \varnothing $,
            $ \Yspan (l) \cap \Yspan (m) \neq \varnothing $. Thus
            necessarily one of them has some points inside the
            territory of the other, and thus by Constraint (iv), is
            entirely inside the territory of the other.
          \end{proof}
          
          We denote by $\score(l)$ the length of the interval
          $\Yspan(l)$ and by $\antiscore(l)$ the maximum real number
          $x$ such that $x$ is the $x$-coordinate of a point of $l$, i.e.\ the maximum of $\Xspan(l)$.

          \begin{lemma}
            \label{l:scoreLine}
            If $l \desc k$, then $\score(l) < \score(k)$ and
            $\antiscore(l) < \antiscore(k)$.  If $l \adj k$, then
            $\score(l) < \score(k)$ and
            $\antiscore(l) > \antiscore(k)$.
          \end{lemma}  

          \begin{proof}
            If $l \desc k$, then $\score(l) < \score(k)$ and
            $\antiscore(l) < \antiscore(k)$ because $ l $ is in the
            territory of $ k $. The inequalities are strict because of
            Constraint~(iv) of Definition~\ref{def:slsg}.

            If $l \adj k$, then $\score(l) < \score(k)$ and
            $\antiscore(l) > \antiscore(k)$ because of Constraint~(vi)
            of Definition~\ref{def:slsg}.
          \end{proof}

          \begin{lemma}
            \label{lem:L-is-BSet}
            $ (L, \desc, \adj) $ forms a Burling set.
          \end{lemma}  

\begin{proof}
  First, $ \desc $ is obviously transitive and asymmetric, so it is a
  strict partial order. Moreover, by Lemma~\ref{l:scoreLine}, the
  relation $ \adj $ cannot have any directed cycle. Now we prove that the four
  axioms of Burling sets hold.

  Axiom \ref{item:descdesc}: If $ k \desc l $ and $ k \desc m $, then
  because of Constraint (v) of Definition~\ref{def:slsg}, $ l $ and
  $ m $ do not intersect. Moreover, because $ k $ is inside the
  territory of both $ l $ and $ m $, then by  Lemma~\ref{l:remark}, one of
  them is inside the territory of the other.
	
  Axiom \ref{item:adjadj}: If $ k \adj l$ and $ k \adj m $, then by
  Constraint (iii), $ m $ and $ l $ do not intersect. Moreover, notice
  that by Constraint (vi), the leftmost point of $ k $ (the lower
  endpoint of $ k $) is inside the territory of both $ l $ and $m
  $. Thus by Lemma~\ref{l:remark}, one of them is inside the territory of
  the other.
	
  Axiom \ref{item:adjdesc}: If $ k \adj l $ and $ k \desc m $, then by
  Lemma~\ref{l:scoreLine},
  $\antiscore(l) < \antiscore(k) < \antiscore(m)$.  So, if $l$ and $m$
  intersect, then by Lemma~\ref{l:scoreLine} again, $ m \adj l $.  So,
  $k$, $l$ and $m$ contradict Constraint (vii).  Hence, $ l $ and
  $ m $ do not intersect. So, the segment $ k $ and thus the
  intersection of $ k $ and $ l $ is in the territory of $ m $, so by
  property (iv), $ l $ is in the territory of $m $, i.e.\
  $ l \desc m $.
	
  Axiom \ref{item:transitiveboth}: If $ k \adj l $ and $ l \desc m $,
  then two cases are possible. Case 1: $ k $ and $ m $ do not
  intersect. Let $ p = (x, y ) $ denote the intersection point of
  $ k $ and $ l $. Because $ p $ is inside the territory of $ m $, by
  Constraint (iv), $ k \desc m $. Case 2: $ k $ and $ m $
  intersect. Then, by Lemma~\ref{l:scoreLine}, $s(k) < s(l) < s(m)$.
  So $ k \adj m $.
	
\end{proof}

A \emph{strict line segment graph} is the intersection graph of a strict set
of line segments in the plane.

\begin{lemma}
  \label{lem:Burling-is-slsg}
  Every Burling graph is a strict line segment graph.
\end{lemma}
\begin{proof}
  In \cite{Pawlik2012Jctb}, graphs of the Burling sequence as
  presented as line segment graphs. One can check easily that in this
  construction, all the constraints of Definition \ref{def:slsg}
  hold. Now, because every Burling graph is an induced subgraph of a
  graph in the Burling sequence, and because removing line segments
  from a strict set of line segments leaves a strict set of line
  segments, the proof is complete.
\end{proof}

\begin{theorem} \label{thm:strict-line-segments=BG}
  The class of strict line segment graphs is equal to the class of
  Burling graphs.
\end{theorem}
      
      \begin{proof}
        By Lemma \ref{lem:L-is-BSet}, every line segment graph, is the
        underlying graph of an abstract Burling graph, and therefore
        by Theorems \ref{lem:abstractB=derived=Burling}, a Burling
        graph. This, along with Lemma \ref{lem:Burling-is-slsg},
        completes the proof.
\end{proof}

\subsection*{Strict box graphs}

Let $S$ be a strict set of frames in the $\mathbb R^2$.  Suppose that
to each frame $A\in S$ is associated a non-empty interval
$I_A$ of $\mathbb R$.  The set of intervals is \emph{compatible
  with $S$} if for all pairs $A, B \in S$ we have:

\begin{itemize}
\item if $A$ enters $B$, then $ I_B \subsetneq I_A$ and
\item if $A$ is inside $B$ then $I_A \cap I_B = \varnothing$.
\end{itemize}

Note that if $A$ and $B$ are incomparable, then there is no condition
on $I_A$ and $ I_B$ (in particular, their intersection can be empty or
not).

\begin{lemma}
  \label{l:interExist}
  For every finite strict set $S$ of frames in the plane, there exists
  a set of intervals compatible with $S$.
\end{lemma}

\begin{proof}
  By lemma~\ref{lem:Frames-are-Abstract-B}, the intersection graph of
  $S$ is an abstract Burling graph. Hence by
  Lemma~\ref{lem:abstractB-is-derived}, $G$ can be derived from a Burling
  tree.  So, by Lemma~\ref{l:chooseT}, $G$ is isomorphic to a graph
  $H$ derived from a Burling tree $ (T, r, \lastBorn, \choosePath) $
  such that $ r $ is not in $V(H)$, every non-leaf vertex in $T$ has
  exactly two children, and no last-born of $T$ is in $V(G)$.  So,
  every frame $A$ of $S$ corresponds to a vertex $v_A\neq r$ of $H$
  that is not a last-born.  Moreover, $A$ is inside $B$ if and only
  if $v_A$ is a descendant of $v_B$ in $T$ and $A$ enters $B$ if and
  only if $v_B \in c(v_A)$.

   Hence, it is enough to prove that we may associate to every vertex $v$
   of $H$ an interval $I_z$ in such a way that for all
   vertices $u, v$ of $H$:

 \begin{itemize}
 \item if $u\in c(v)$, then
   $I_u \subsetneq  I_v$ and
 \item if $u$ is a proper descendant of $v$ then
   $I_u \cap I_v = \varnothing$.
 \end{itemize}

 We now define the intervals.  We first perform a DFS search of $T$,
 starting at the root and giving priority to the last-borns.  This
 defines an integer $f(v)$ for each vertex $v$ of $T$ satisfying
 $f(r) = 1$, and for every non-leaf vertex $v$ with last-born child
 $u$ and non-last-born child $w$, $f(u) = f(v)+1$ and
 $$f(w) = \max\{f(x): \text{ $x$ is a descendant of $u$}\} + 1.$$ See
 Figure~\ref{fig:inter}.

         \begin{figure}
          \centering

          \includegraphics[height=3.5cm]{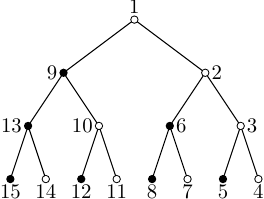} \hspace{1cm}
          \includegraphics[height=3.3cm]{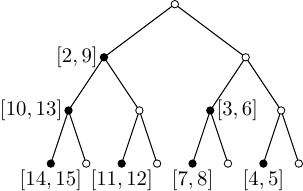}
          \caption{Intervals associated to the non-last-borns of a
            Burling tree\label{fig:inter}.} 
        \end{figure}

        Let $w$ be a vertex of $H$. So, $w$ is a vertex of $T$ that is
        neither $r$ nor a last-born. It follows that $u$ has a parent
        $v$ that has a last-born child $u$.  We associate to $w$ the
        interval $[f(u), f(w)]$ (note that $f(u) < f(w)$ since we
        apply DFS with priority to the last-borns).

        Suppose that $w$ is a proper descendant of $w'$, and their
        intervals are $[f(u), f(w)]$ and $[f(u'), f(w')]$ with
        notation as above.  In fact, both $u$ and $w$ are descendant
        of $w'$, so by the properties of DFS, $f(u) > f(w')$. This
        implies that $[f(u), f(w)]$ and $[f(u'), f(w')]$ are disjoint.

        Suppose that $w'\in c(w)$, and their intervals are
        $[f(u), f(w)]$ and $[f(u'), f(w')]$ with notation as above.
        Note that $u'$ and $w'$ are both descendants of $u$. So
        $f(u) < f(u') < f(w')$. And since $w'$ is a descendant of $u$,
        $f(w') < f(w)$.  Hence
        $[f(u'), f(w')] \subsetneq [f(u), f(w)]$.
      \end{proof}

      When an interval $I$ is associated to a frame $A$ of
      $\mathbb R^2$, there is a natural way to define an axis-align
      box of $\mathbb R^3$:
      $\{(x, y, z): (x, y)\in A \interior, z \in I\}$. This is the box
      \emph{associated} to $A$ and $I$.
 
 A set of axis-aligned boxes of $\mathbb R^3$ is \emph{strict} if it
 can be obtained from a strict set $S$ of frames by considering of set of
 intervals compatible with $S$, and by taking for each frame $A$ and
 each interval $I_A$ the box associated to $A$ and $I_A$.

 \begin{lemma}
   \label{2d3dEquiv}
   Suppose that a strict set $S'$ of boxes is obtained from a strict
   set of frames $S$.  Let $A, B \in S$ be frames, and $A', B'$ be the
   respective boxes associated to them.  Then
   $A\cap B \neq \varnothing$ if and only if
   $A' \cap B' \neq \varnothing$.  In particular, the intersection
   graph of $S$ is isomorphic to the intersection graph of $S'$.
 \end{lemma}

 \begin{proof}
   If $A$ enters $B$, then $A'\cap B'\neq \varnothing$ because both
   the frames and the interval associated to them have a non-empty
   intersection. If $A$ is inside $B$, then $A'\cap B' = \varnothing$
   because the intervals associated to $A$ and $B$ are disjoint.  If
   $A$ and $B$ are incomparable, then $A'\cap B' = \varnothing$
   because $A\interior \cap B\interior = \varnothing$.
 \end{proof}

A \emph{strict box graph} is the intersection graph of a strict set
 of boxes of $\mathbb R^3$.
 
 \begin{theorem}
   The class of strict box graphs is equal to the class of Burling
   graphs.
 \end{theorem}

 \begin{proof}
   Suppose that $G$ is a Burling graph.  Then, by
   Theorem~\ref{thm:strict-frames=BG}, $G$ is the intersection graph
   of a strict set $S$ of frames.  By Lemma~\ref{l:interExist}, a set
   of intervals compatible with $S$ exists. Hence, by
   Lemma~\ref{2d3dEquiv}, $G$ is isomorphic to a strict box graph.

   Suppose conversely that $G$ is a strict box graph. Then, by
   definition, it arises from a strict set of frames and a set of
   interval compatible with it.  Hence, by Lemma~\ref{2d3dEquiv}, $G$
   is isomorphic to a strict frame graph. So by
   Theorem~\ref{thm:strict-frames=BG}, $G$ is a Burling graph.
 \end{proof}
 
\section*{Acknowledgment}

Thanks to an anonymous referee for helping us to clarify the bibliography
and to Louis Esperet, Gwenaël Joret and Paul Meunier for
useful discussions.



\providecommand{\bysame}{\leavevmode\hbox to3em{\hrulefill}\thinspace}
\providecommand{\MR}{\relax\ifhmode\unskip\space\fi MR }
\providecommand{\MRhref}[2]{%
  \href{http://www.ams.org/mathscinet-getitem?mr=#1}{#2}
}
\providecommand{\href}[2]{#2}

\end{document}